\documentclass[a4paper,11pt]{amsart}
\usepackage{hyperref} %

\usepackage{lmodern}
\usepackage[T1]{fontenc}
\usepackage{amsmath, amsthm, amssymb, mathrsfs}
\usepackage[matha,mathx]{mathabx}

\numberwithin{equation}{section}
\theoremstyle{plain}
 \newtheorem{thm}{Theorem}[section]
 \newtheorem{cor}[thm]{Corollary}
 \newtheorem{lem}[thm]{Lemma}
 \newtheorem{prop}[thm]{Proposition}
 
\theoremstyle{definition}

\theoremstyle{remark}
 \newtheorem{rem}[thm]{Remark}

\allowdisplaybreaks[1]

\newcommand{\bsk}{{\boldsymbol{k}}}
\renewcommand{\l}{\ell}
\newcommand{\R}{\mathcal{R}}
\newcommand{\Rp}{\R^+}
\newcommand{\C}{{\mathbb{C}}}
\newcommand{\B}{\mathcal{B}}
\newcommand{\cc}{\mathbf{c}}
\DeclareMathOperator*{\Res}{Res}
\newcommand{\acs}{\mathfrak{a}_\mathbb{C}^*}
\newcommand{\al}{{\boldsymbol{\alpha}}}
\newcommand{\be}{{\boldsymbol{\beta}}}
\newcommand{\simarrow}{\xrightarrow{\smash[b]{\lower 0.7ex\hbox{$\sim$}}}}
\newcommand{\Dki}{D_\bsk(\varTheta_i)}
\newcommand{\Dko}{D_\bsk(\varTheta_1)}
\newcommand{\Dkr}{D_\bsk(\B)}
\newcommand{\Dkz}{D_\bsk(\emptyset)}
\newcommand{\PW}{\mathcal{PW}(\mathfrak{a}^*_\mathbb{C})}
\newcommand{\Cl}{\operatorname{Cl}}
\newcommand{\sgn}{\operatorname{sgn}}
\AtBeginDocument{
  \abovedisplayskip     =0.4\abovedisplayskip
  \abovedisplayshortskip=0.4\abovedisplayshortskip
  \belowdisplayskip     =0.4\belowdisplayskip
  \belowdisplayshortskip=0.4\belowdisplayshortskip}

\title[Inversion formula for the hypergeometric Fourier transform]{
Inversion formula for the 
hypergeometric Fourier transform 
associated with a 
root system of type $BC$}
\author{Tatsuo \textsc{Honda}}
\address{Learning Support Center, Faculty of Engineering, Takushoku University,
815-1 Tatemachi, Hachioji, Tokyo 193-0985, Japan}
\email{thonda@ner.takushoku-u.ac.jp}
\author{Hiroshi \textsc{Oda}}
\address{Faculty of Engineering, Takushoku University,
815-1 Tatemachi, Hachioji, Tokyo 193-0985, Japan}
\email{hoda@la.takushoku-u.ac.jp}
\author{Nobukazu \textsc{Shimeno}
}
\address{School of Science, Kwansei Gakuin University, 
1 Gakuen Uegahara, Sanda, Hyogo 669-1330, Japan}
\email{shimeno@kwansei.ac.jp}


\begin{document}
\begin{abstract}
We give the inversion formula and the Plancherel formula for the hypergeometric 
Fourier transform associated with a root system of type $BC$, when the multiplicity 
parameters are not necessarily non-negative. 
\end{abstract}

\maketitle

\section*{Introduction}

Heckman and Opdam \cite{HO87, Hec87,Hec94, Hec97, {Op:Cherednik}, Op2000,HO2020}
 have developed a theory of hypergeometric functions 
associated with root systems. When the multiplicity function $\bsk$ takes some particular values, the 
Heckman-Opdam hypergeometric function coincides with the restriction to a 
Cartan subspace $\mathfrak{a}$  of the zonal spherical 
function on a Riemannian symmetric space of the non-compact type. 

In this group case, 
the associated harmonic analysis have been developed by Harish-Chandra, Gindikin-Karpelevich, 
Helgason, Gangolli, Rosenberg, and other researchers (cf. \cite{GV, Hel2}). In particular, an explicit 
inversion formula, the 
Paley-Wiener theorem, and the Plancherel theorem for the 
spherical Fourier transform have been established. 

Opdam \cite{Op:Cherednik} generalized these results 
to the hypergeometric Fourier transform,
where the multiplicity function may take arbitrary non-negative values.
In this case there are only continuous spectra.
More precisely, the Plancherel measure is $\cc(\lambda,\bsk)^{-1}
\cc(-\lambda,\bsk)^{-1}d\mu(\lambda)$ on $\sqrt{-1}\mathfrak{a}^*$,
where $\cc(\lambda,\bsk)$ 
is Harish-Chandra's 
$\cc$-function and $\mu$ is a normalized Lebesgue measure. 
In \cite{Op99}
Opdam also studied the case of negative multiplicity functions when the root system 
is reduced.
In such a situation, spectra with supports of lower dimensions appear in the Plancherel measure, in addition to the most continuous spectra described above. 

If the root system is of type $BC_1$, then the associated Heckman-Opdam hypergeometric 
function is nothing but the Jacobi function, which was introduced and studied by Flensted-Jensen and 
Koornwinder \cite{FJ77, Ko75, Ko84}. 
The Jacobi function is an even eigenfunction of a second order ordinary differential operator 
and can be written by the Gauss hypergeometric function. 
When the multiplicity function corresponds to
a rank one Riemannian symmetric space of the noncompact type, 
the second order differential operator is the radial part of the Laplace-Beltrami 
operator and the Jacobi function is 
the restriction  of the zonal spherical 
function to a 
Cartan subspace $\mathfrak{a}\simeq \mathbb{R}$. 
Let $\bsk_s$ and $\bsk_\ell$ denote the 
multiplicity parameters for the short roots $\pm\beta$ and the long roots $\pm2\beta$, respectively. 
Flensted-Jensen \cite[Appendix~1]{FJ77} proved the inversion formula for 
the Jacobi transform (the hypergeometric Fourier transform in the $BC_1$ case)
under the conditions $\bsk_s,\,\bsk_\ell \in\mathbb{R}$ and $\bsk_s+\bsk_\ell>-\frac12$.
If $\bsk$ further satisfies
$\bsk_s+\bsk_\ell+\frac12-\big|\bsk_\ell-\frac12\big|\geq 0$, then only continuous 
spectra appear. In the other cases, finite series of discrete spectra appear in addition to the continuous spectra.
These discrete spectra and the corresponding Plancherel measure are obtained from 
residues of $\cc(\lambda,\bsk)^{-1}\cc(-\lambda,\bsk)^{-1}$. 
Flensted-Jensen \cite{FJ77} applied his results on the Jacobi transform
to harmonic analysis of spherical functions 
on the universal covering group of $SU(n,1)$ associated with a one-dimensional $K$-type. 
The Jacobi transform also can be applied to the harmonic analysis on some 
homogeneous vector bundles over hyperbolic spaces (cf. \cite{DP, Sh95}). 
In these group cases, discrete spectra correspond to relative discrete series representations. 

The Heckman-Opdam hypergeometric function is a real analytic joint eigenfunction 
of a commuting family of differential operators. 
In a group case, the differential operators are radial parts of the invariant 
differential operators. The inversion formula  for the hypergeometric 
Fourier transform gives an expansion of an arbitrary Weyl group invariant 
function in terms of the Heckman-Opdam 
hypergeometric functions. Heckman-Opdam theory gives a simultaneous generalization of the Euclidean 
Fourier analysis (the case of $\bsk\equiv 0$), 
the theory of spherical functions, and the Jacobi analysis, and it provides a rich 
framework of harmonic analysis associated with root systems. 

In this paper we will study the case of the root system of type $BC$ and arbitrary rank 
when the multiplicity 
function is not necessarily non-negative. 
Except for the case of type $BC_1$  
mentioned above or a group case studied by the third author \cite{Sh94}, 
this study has remained open for many years. 
We have decided to study this case
for its application to harmonic analysis of 
spherical functions associated with certain $K$-types on connected semisimple Lie 
groups of finite center (cf. \cite[final comment]{OdSh}). 

For the root system of type $BC_r$ with $r\geq 2$,
 there are three multiplicity parameters $\bsk_s,\,\bsk_m$ and $\bsk_\ell$ 
corresponding to the short, medium, and long roots respectively. 
We prove the inversion formula, the Paley-Wiener theorem, and 
the Plancherel theorem for the hypergeometric Fourier transform under the conditions
$\bsk_s,\,\bsk_m,\,\bsk_\ell\in \mathbb{R}$,
$\bsk_s+\bsk_\ell>-\frac12$ and $\bsk_m\geq 0$. 
These conditions on $\bsk$ make the spectral problem well-posed and cover the known group cases. 
An explicit expression of the Plancherel measure is obtained
by calculus of residues of $\cc(\lambda,\bsk)^{-1}\cc(-\lambda,\bsk)^{-1}$.
In particular, the square 
integrable hypergeometric functions are classified and
their $L^2$-norms are calculated explicitly. 
Our method using residue calculus
closely follows that of \cite{Sh94}, 
where the third author  studied 
spherical functions associated with a one-dimensional $K$-type 
on an irreducible Hermitian symmetric space and obtained the inversion formula for the 
spherical transform. 

This paper is organized as follows.
In Section \ref{sect:HP} we review the Heckman-Opdam hypergeometric 
function. In Section \ref{sect:ft}
we define the hypergeometric Fourier transform associated with a root system of type $BC$ for various $\bsk$
and derive a first form of inversion formula
(Theorem~\ref{thm:1st} and Theorem~\ref{thm:cont}) from Opdam's result for the non-negative multiplicity functions. 
In Section~\ref{sect:tempered} we define
some sets of tempered spectra \eqref{eqn:D}
that are crucial to the main results of this paper
and study 
tempered and square integrable hypergeometric functions.
In Section~\ref{sect:partial} we introduce a partial sum of Harish-Chandra series
in order to simplify the residue calculus in the subsequent section.
In Section~\ref{sect:inversion}
we prove the final form of inversion formula 
(Theorem~\ref{thm:main}), the Paley-Wiener theorem (Theorem~\ref{thm:pw}), 
the Plancherel theorem (Theorem~\ref{thm:plancherel}),
and give the classification of square integrable hypergeometric functions
and their $L^2$-norms explicitly (Corollary~\ref{cor:lthg}).
Explicit formulas of 
the residues of $\cc(\lambda,\bsk)^{-1}\cc(-\lambda,\bsk)^{-1}$ that contribute to the 
Plancherel measure are given by \eqref{eqn:fdefd}, \eqref{eqn:fdefd2} and Proposition~\ref{prop:dtheta}. 

\section{
Heckman-Opdam hypergeometric functions
}\label{sect:HP}

In this section, we give basic notations for a root system of type $BC$ and review on 
the Heckman-Opdam hypergeometric function associated with a root system of type  $BC$. 
We refer the reader to original papers by Heckman and Opdam \cite{Hec87, HO87, Op:Cherednik} and 
survey articles \cite{Hec94, Hec97, Op2000, An, HO2020} for details.

Let $\mathbb{N}$ denote the set of non-negative integers. Let $r$ denote a positive integer and 
$\mathfrak{a}$ an $r$-dimensional Euclidean space with an inner product $\langle\,\cdot,\,\cdot\,\rangle$. 
We often identify $\mathfrak{a}$ and $\mathfrak{a}^*$ by using the inner product. 
We use the same notation $\langle\,\cdot,\,\cdot\,\rangle$ for the inner product 
on $\mathfrak{a}^*$ and the complex symmetric bilinear form 
on $\acs\times \acs$. 
For $\lambda\in \acs$ define $||\lambda||=\sqrt{\langle \lambda,\bar\lambda\rangle}$, 
where $\bar{\,\cdot\,}$ denotes complex conjugation. 
Let $\R\subset \mathfrak{a}^*$ be a root system of type $BC_r$. 
It is of the form
\begin{equation}
\R=\{\pm \beta_i,\,\pm 2\beta_i,\,\pm(\beta_p\pm \beta_q)\,;\,1\leq i\leq r,\,1\leq q<p\leq r\}, \\
\end{equation}
where $\{\beta_1,\beta_2,\dots,\beta_r\}$ is an orthogonal basis of $\mathfrak{a}^*$ with $||\beta_p||=||\beta_q||\,\,
(1\leq q<p\leq r)$. Moreover, 
we assume $||\beta_1||=2$, so $\{\frac12\beta_1,\dots,\frac12\beta_r\}$ 
forms an orthonormal basis of $\mathfrak{a}^*$. 
Let $\Rp$ denote the positive system of $\R$ defined by 
\begin{equation}
\Rp=\{\beta_i,\,2\beta_i,\,\beta_p\pm \beta_q\,;\,1\leq i\leq r,\,1\leq q<p\leq r\}.
\end{equation}
Let $\alpha_1,\dots,\alpha_r$ denote the positive roots defined by 
\begin{equation}
\alpha_i=\beta_{r+1-i}-\beta_{r-i}\,\,\,\,(1\leq i\leq r-1),\quad \alpha_r=\beta_1.
\end{equation}
Then 
\begin{equation}\label{eqn:sr}
\B=\{\alpha_1,\alpha_2,\dots,\alpha_r\}
\end{equation}
is the set of the simple roots in $\Rp$. 
We call $\pm \beta_i,\,\pm(\beta_p\pm \beta_q)\,\,(p<q),$ and $\pm 2\beta_i$,  short, medium, and long 
roots, respectively. 
Let $\Rp_s,\,\Rp_m$, and $\Rp_\l$ denote the subset of $\Rp$ consisting of 
the short, medium, and long roots, respectively. 
Let $W$ denote the Weyl group for $\R$. 
It is the semidirect product of $\mathbb{Z}_2^r$ and $\mathfrak{S}_r$. 
The group $\mathbb{Z}_2^r$ and $\mathfrak{S}_r$ act on $\mathfrak{a}^*$  as sign changes 
and permutations of $\{\beta_1,\dots,\beta_r\}$, respectively. 

Let $\bsk$ be a complex-valued $W$-invariant function on $\R$, which is called a multiplicity function.
Let $\mathcal{K}_\mathbb{C}$ be the space of the multiplicity functions. 
Let $\bsk_s,\,\bsk_m$, and $\bsk_\l$ be the values of $\bsk$ for the short, medium, 
and long roots respectively. 
We identify $\bsk\in \mathcal{K}_\mathbb{C}$ with the 3-tuple $(\bsk_s,\bsk_m,\bsk_\l)\in\C^3$, so
$
\mathcal{K}_\C\simeq\mathbb{C}^3
$. 
Let 
\[
\rho(\bsk)=\tfrac12\!\!\sum_{\alpha\,\in\,\Rp}\!\!\bsk_\alpha\, \alpha.
\]

For $\alpha\in\R$ let $\alpha^\vee$ denote the corresponding coroot 
$\alpha^\vee=2\alpha/\langle\alpha,\alpha\rangle$. 
Any $\lambda\in\acs$ can be written in the form
\begin{equation}\label{eqn:coordas}
\lambda=\tfrac12\sum_{i=1}^r \lambda_i \beta_i\quad\text{with}\quad 
\lambda_i=\langle\lambda,\beta_i^\vee\rangle\,\,\,\,(1\leq i\leq r).
\end{equation}
We identify $\lambda\in\acs$ with the $r$-tuple $(\lambda_1,\dots,\lambda_r)\in\C^r$, 
so $\acs\simeq\C^r$. 
In particular, 
\begin{equation}
\rho(\bsk)=(\bsk_s+2\bsk_\l, \bsk_s+2\bsk_\l+2\bsk_m,\dots,\bsk_s+2\bsk_\l+2(r-1)\bsk_m).
\end{equation}

For $\bsk\in\mathcal{K}_\C$ let $L(\bsk)$ denote the differential operator on $\mathfrak{a}$ 
defined by
\begin{equation}
L(\bsk)=\frac14\sum_{i=1}^r \partial_{\beta_i}^2+\sum_{\alpha\,\in\, \Rp}\!\!\bsk_\alpha\coth\frac{\alpha}{2}\,\,\partial_\alpha.
\end{equation} 
Here 
$\partial_\alpha$ denotes the directional derivative corresponding to $\alpha$. 
There exists an algebra $\mathbb{D}(\bsk)$ of $W$-invariant differential operators on $\mathfrak{a}$ 
with the properties: $L(\bsk)\in \mathbb{D}(\bsk)$ and 
there exists an algebra isomorphism $\gamma_\bsk:\mathbb{D}(\bsk)\simarrow S(\mathfrak{a}_\C)^W$, 
where $S(\mathfrak{a}_\C)^W$ denotes the set of the $W$-invariant elements in the symmetric 
algebra $S(\mathfrak{a}_\C)$. 
For $\lambda\in\acs$ we consider the system of differential equations
\begin{equation}\label{eqn:hgs}
D f=\gamma_\bsk(D)(\lambda)f\quad\text{for any }\,D\in\mathbb{D}(\bsk).
\end{equation}
In particular, it contains the equation
\begin{equation}\label{eqn:lapeq}
L(\bsk)f= (\langle\lambda,\lambda\rangle-\langle\rho(\bsk),\rho(\bsk)\rangle)f.
\end{equation}

Let
\[
 \mathfrak{a}_+
=\{x\in\mathfrak{a}\,;\,\alpha(x)>0 \,\,\text{ for all } \alpha\in\Rp\}.
\]
Let $Q_+$ denote the subset of $\mathfrak{a}^*$ spanned by $\B$ over $\mathbb{N}$. 
There exists a series solution $\varPhi(\lambda,\bsk)$ of \eqref{eqn:hgs} of the form
\begin{equation}\label{eqn:hcseries}
\varPhi(\lambda,\bsk;x)=\!\sum_{\kappa\,\in\, Q_+}\!\!a_\kappa(\lambda,\bsk)\,e^{(\lambda-\rho(\bsk)-\kappa)(x)}
\quad (x\in \mathfrak{a}_+),
\end{equation}
with $a_0(\lambda,\bsk)=1$. For a generic $\lambda$ all the coefficients $a_\kappa(\lambda,\bsk)\,\,\,(\kappa\in Q_+)$ 
are determined uniquely by \eqref{eqn:lapeq} and $\varPhi(\lambda,\bsk)$ converges on $\mathfrak{a}_+$. 
We call  $\varPhi(\lambda,\bsk)$ the Harish-Chandra series. 
As a function of the spectral parameter $\lambda\in\acs$, $\varPhi(\lambda,\bsk)$ is 
meromorphic with simple poles along hyperplanes of the form
\begin{equation}\label{eqn:condlambda0}
\langle\lambda,\alpha^\vee\rangle=j\quad\text{for some }\,\alpha\in\Rp, \,\,j=1,2,3,\dots.
\end{equation}
See \cite[Proposition~4.2.5]{Hec94}.  
Moreover, if 
\begin{equation}\label{eqn:condlambda}
\langle\lambda,\alpha^\vee\rangle\not\in \mathbb{Z}\quad\text{for all }\,\alpha\in\R,
\end{equation}
then $\{\varPhi(w\lambda,\bsk)\,;\,w\in W\}$ forms a basis of the solution space of \eqref{eqn:hgs} on 
$\mathfrak{a}_+$. 
The condition \eqref{eqn:condlambda} can be written explicitly as
\begin{equation}\label{eqn:condlambda2}
\lambda_i\not\in\mathbb{Z}\,\,\,(1\leq i\leq r),\quad 
\frac{\lambda_p\pm \lambda_q}2\not\in\mathbb{Z}\,\,\,(1\leq p\not=q\leq r).
\end{equation}

Define
\begin{align*}
 \mathfrak{a}_+^*
&=\{\lambda\in\mathfrak{a}^*\,;\,\langle\lambda,\alpha\rangle>0 \,\,\text{ for all } \alpha\in \Rp\} \\
& \simeq\{(\lambda_1,\lambda_2,
\dots,\lambda_r)\in\mathbb{R}^r\,;\,
0<\lambda_1<\lambda_2<\cdots<\lambda_r\}.
\end{align*}
Let $\Cl(\mathfrak{a}_+^*)$ denote the closure of ${\mathfrak{a}_+^*}$.

\begin{lem}\label{lem:esthcs}
Suppose $T\subset \mathfrak{a}^*$ and $U\subset\mathcal{K}_\C\simeq \C^3$ are compact sets. 
Define a polynomial $p(\lambda)$
to be the product of $\langle2\lambda-\kappa,\kappa\rangle$
taken over $\kappa\in Q_+\setminus\{0\}$ such that $\langle2\lambda-\kappa,\kappa\rangle= 0$
for some $\lambda \in T$.
(Note that this is a finite product.)
Then for any $\varepsilon>0$
there exist constants $C>0$ and $n\in\mathbb{N}$ such that
\[
|p(\lambda)a_{\kappa}(\lambda,\bsk)|\le C(1+||\lambda||)^n e^{\kappa(x)}
\]
whenever $\kappa\in Q_+$, $\operatorname{Re}\lambda\in T$, $\bsk\in U$
 and 
$x\in \mathfrak{a}$ satisfies $\alpha(x)>\varepsilon$ for any $\alpha\in\B$. 
Hence for any compact subset $V\subset \mathfrak a_+$ and any $q\in S(\mathfrak a_\C)$
there exist constants $C'>0$ and $n'\in\mathbb{N}$ such that 
\[
|q(\partial_x)p(\lambda)\varPhi(\lambda,\bsk;x)|\leq C'(1+||\lambda||)^{n'}
\]
for $(\lambda,\bsk,x)\in(T+\sqrt{-1}\mathfrak a^*)\times U\times V$.
\end{lem}
\begin{proof}
The lemma follows from a natural extension of
the estimates for the coefficients $a_\kappa(\lambda,\bsk)$ 
due to Gangolli. See \cite[Ch.~IV Lemma~5.6]{Hel2} and \cite[Theorem~4.5.4]{GV}. 
See also \cite[Ch.~I Lemma~5.1]{Ar} and \cite[Corollary~3.11]{Sh94}. 
\end{proof}

Let 
 $\varGamma(\,\cdot\,)$ denote the Gamma function. 
Define the meromorphic functions $\tilde{\cc}_\alpha(\lambda,\bsk)\,\,(\alpha\in\R)$, 
$\tilde{\cc}(\lambda,\bsk)$, and $\cc(\lambda,\bsk)$ on 
$\acs\times\mathcal{K}_\C$ by 
\begin{align}
& \tilde{\cc}_\alpha(\lambda,\bsk)=
\frac{\varGamma\big(\langle\lambda,\alpha^\vee\rangle+\frac12\bsk_{\frac12\alpha}\big)}{
\varGamma\big(\langle\lambda,\alpha^\vee\rangle+\frac12\bsk_{\frac12\alpha}+\bsk_\alpha\big)}, \label{eqn:tca}\\
& \tilde{\cc}(\lambda,\bsk)   =\!\prod_{\alpha\,\in\, \Rp} \!\!\tilde{\cc}_\alpha(\lambda,\bsk),\label{eqn:tc}
\end{align}
and 
\begin{equation}\label{eqn:cf}
\cc(\lambda,\bsk)=  \frac{\tilde{\cc}(\lambda,\bsk)}{\tilde{\cc}(\rho(\bsk),\bsk)}
\end{equation}
with the convention $\bsk_{\frac12\alpha}=0$ if $\frac12\alpha\not\in \R$. 
We call $\cc(\lambda,\bsk)$ Harish-Chandra's $\cc$-function. 
By \cite[(3.4.6)]{Hec94} and \cite[Proposition~5.1]{Hec97}, 
\begin{align}
\tilde{\cc}(\lambda,\bsk)   = &
\prod_{1\leq q<p\leq r}\frac{\varGamma\big(\frac12(\lambda_p-\lambda_q)\big)\varGamma\big(\frac12(\lambda_p+\lambda_q)\big)
}{\varGamma\big(\frac12(\lambda_p-\lambda_q+2\bsk_m)\big)\varGamma\big(\frac12(\lambda_p+\lambda_q+2\bsk_m)\big)} 
\label{eqn:tcf} \\
& 
\times  \prod_{i=1}^r \frac{2^{-\bsk_s}
\varGamma\big(\frac12\lambda_i\big)\varGamma\big(\frac12(\lambda_i+1)\big)}
{\varGamma\big(\frac12(\lambda_i+\bsk_s+1)\big)\varGamma\big(\frac12(\lambda_i+\bsk_s+2\bsk_\l)\big)}
 \notag 
\end{align}
and 
\begin{equation}\label{eqn:crho}
\tilde{\cc}(\rho(\bsk),\bsk)=\prod_{i=1}^r \frac{\varGamma(\bsk_s+(i-1)\bsk_m+\bsk_\l)\varGamma(\bsk_m)}
{\varGamma(2(\bsk_s+(i-1)\bsk_m+\bsk_\l))\varGamma(i\bsk_m)}.
\end{equation}
The latter is regular on the entire $\mathcal K_\C$.
Note 
\begin{equation}
\cc(\lambda,0)=\frac{1}{2^r r!}=\frac{1}{|W|} \quad\text{for}\quad \bsk_s=\bsk_m=\bsk_\l=0. 
\end{equation}

Let $\mathcal{K}_\text{reg},\,\mathcal{K}$, and $\mathcal{K}_+$ denote 
the subsets of $\mathcal{K}_\C\simeq \C^3$ given by
\begin{align*}
& \mathcal{K}_\text{reg}=\{\bsk\in\mathcal{K}_\C\,;\,\tilde{\cc}(\rho(\bsk),\bsk)\not=0\}, \\
& \mathcal{K}=\{\bsk\in\mathcal{K}_\C\,;\,\bsk_s,\,\bsk_m,\,\bsk_\l\in\mathbb{R}\}, \\
& \mathcal{K}_+=\{\bsk\in\mathcal{K}\,;\,\bsk_\alpha\geq 0\,\,\text{ for all }\alpha\in \R\}.
\end{align*}
Thus $\mathcal{K}_+\subset \mathcal{K}_\text{reg}$.

Throughout the paper we repeatedly use the following type of estimate
which is applicable to $\cc(-\lambda,\bsk)^{-1}$ and other similar functions.
\begin{lem}\label{lem:estcf}
Let $\{(v_j,a_j(\bsk),b_j(\bsk))\}_{j=1}^k$ be a sequence of triples
consisting of $v_j\in \mathfrak a\,\setminus\,\{0\}$ and polynomials $a_j(\bsk), b_j(\bsk)$ in $\bsk$.
Suppose $T\subset \mathfrak a^*$ and $U\subset \mathcal{K}_\C$ are compact sets such that
\[
\psi(\lambda,\bsk)=\prod_{j=1}^k\frac{\varGamma(\lambda(v_j)+a_j(\bsk))}{\varGamma(\lambda(v_j)+b_j(\bsk))}
\]
is regular on $(T+\sqrt{-1}\mathfrak a^*)\times U$.
Then there exists constants $C>0$ and $n\in\mathbb N$ such that
\[
|\psi(\lambda,\bsk)|\le C(1+||\lambda||)^n
\]
on $(T+\sqrt{-1}\mathfrak a^*)\times U$.
\end{lem}
\begin{proof}
Choose an arbitrary $v\in\{v_1,\ldots,v_k\}$ and let $\psi_v(\lambda,\bsk)$
be the product of $\varGamma(\lambda(v_j)+a_j(\bsk))/\varGamma(\lambda(v_j)+b_j(\bsk))$ for all $j$ such that $v_j$ is proportional to $v$.
Then by the local theory of meromorphic functions,
$\psi_v(\lambda,\bsk)$ is regular on $(T+\sqrt{-1}\mathfrak a^*)\times U$.
By Stirling's formula
\[
\varGamma(z)=\sqrt{\frac{2\pi}z}\biggl(\frac{z}{e}\biggr)^z\biggl(1+O\biggl(\frac1z\biggr)\biggr)\quad
(|\arg z|<\pi-\delta, |z|\to\infty;\, \forall\delta>0),
\]
there exist constants $C_v>0$ and $n_v\in\mathbb N$ such that
\[
|\psi_v(\lambda,\bsk)|\le C_v(1+|\lambda(v)|)^{n_v}
\]
on $(T+\sqrt{-1}\mathfrak a^*)\times U$.
Since $\psi$ is a product of $\psi_v$'s, we are done.
(See \cite[Ch.~IV, Proposition~7.2]{Hel2}, \cite[Proposition~4.7.15]{GV}, \cite[Lemma~3.9]{Sh94}, 
and \cite[Lemma~2.2]{Ko75}.)
\end{proof}

For $\lambda\in \acs$ satisfying \eqref{eqn:condlambda} and $\bsk\in\mathcal{K}_\mathrm{reg}$, define
\begin{equation}
{F}(\lambda,\bsk;x)=\!\sum_{w\,\in\, W}\!{\cc}(w\lambda,\bsk)\, \varPhi(w\lambda,\bsk;x).
\label{eqn:hcs2}
\end{equation}
The coefficients $\cc(w\lambda,\bsk)$ and terms $\varPhi(w\lambda,\bsk)$ on the right hand side 
are meromorphic function on $\acs\times\mathcal{K}_\C$ and indeed regular for $\lambda$ and $\bsk$ satisfying  
\eqref{eqn:condlambda} 
and $\bsk\in\mathcal{K}_\text{reg}$. A deep theorem due to Heckman and Opdam says that 
$F(\lambda,\bsk;x)$ extends to an 
 analytic function on $\acs\times \mathcal{K}_\text{reg}\times \mathfrak{a}$ 
and $F(\lambda,\bsk;0)=1$. 
It satisfies
\begin{align}
& F(w\lambda,\bsk;x)=F(\lambda,\bsk;x)\quad \text{for all }\,w\in W, \label{eqn:hgprop1}\\
& F(\lambda,\bsk;wx)=F(\lambda,\bsk;x)\quad \text{for all }\,w\in W, \label{eqn:hgprop2}\\
& \overline{F(\lambda,\bsk;x)}=F(\bar{\lambda},\bar{\bsk};x) \label{eqn:hgprop4}.
\end{align}
$F(\lambda,\bsk;x)$ is the unique $W$-invariant real analytic 
solution of the hypergeometric system \eqref{eqn:hgs} on $\mathfrak{a}$ satisfying $F(\lambda,\bsk;0)=1$. 
We call $F(\lambda,\bsk;x)$ the Heckman-Opdam hypergeometric function associated with the root 
system $\R$. 

\begin{rem}\label{rem:rs}
(1) \,The Heckman-Opdam hypergeometric function is a generalization of the 
zonal spherical function on a Riemannian symmetric space of the non-compact type. 
Let $G$ be a connected non-compact semisimple Lie group of finite center 
with the Cartan decomposition $G=K\exp\mathfrak{a} \,K$. 
Let $\varSigma\subset\mathfrak{a}^*$ be the restricted root system for $(\mathfrak{g},\mathfrak{a})$ and $\boldsymbol{m}_\alpha$ 
the dimension of the root space corresponding to $\alpha\in\varSigma$. Set
\[
\R=2\varSigma,\quad \bsk_{\alpha}=\boldsymbol{m}_{\alpha/2}\,\,\,(\alpha\in \R).
\]
Then $L(\bsk)$ is the radial part of the Laplace-Beltrami operator on $G/K$, $\cc(\lambda,\bsk)$ 
is Harish-Chandra's $\cc$-function, $F(\lambda,\bsk)$ is the restriction to $\mathfrak{a}$ of 
the zonal spherical function on $G/K$. 
(Here $\R$ and $\varSigma$ are not  necessarily of type $BC$.)

The zonal spherical function is a bi-$K$-invariant function on $G$. More generally, 
elementary spherical functions associated with some $K$-types are expressed by 
the Heckman-Opdam hypergeometric function. 
The case of one-dimensional $K$-types when $G$ is of  Hermitian  type is 
studied by \cite[Section~5]{Hec94} and \cite{Sh94}. 
More generally, the cases of small $K$-types are studied by \cite{OdSh}. 
We call these cases ``the group case'' collectively. 

\smallskip\noindent
(2) \,
When the root system $\R$ is of type $BC_1$, the Heckman-Opdam hypergeometric function 
is the Jacobi function studied by Flensted-Jensen and Koornwinder:
\begin{align*}
F(\lambda,\bsk;x) & =\phi_{\sqrt{-1}\lambda}^{(\al,\be)}(z) \\
& :={}_2 F_1\big(\tfrac12(\lambda+\rho(\bsk)),\tfrac12(-\lambda+\rho(\bsk));\al+1;-(\sinh z)^2\big), 
\end{align*}
with $z=\frac14\beta_1(x)$, $\al=\bsk_s+\bsk_\l-\frac12,\,\be=\bsk_\l-\frac12$, and 
$\rho(\bsk)=\al+\be+1$, where 
${}_2 F_1$ denote the Gauss hypergeometric function (cf. \cite{FJ77, Ko75, Ko84}). 

\smallskip\noindent
(3) \,
The hypergeometric system \eqref{eqn:hgs} has regular singular points at infinity in $\mathfrak{a}_+$ 
and is holonomic of rank $|W|$. The leading exponents at infinity of \eqref{eqn:hgs} are of the form 
$w\lambda-\rho(\bsk)\,\,(w\in W)$. If $\lambda\in\acs$ satisfies \eqref{eqn:condlambda}, 
then $\varPhi(w\lambda,\bsk)\,\,(w\in W)$ are solutions of \eqref{eqn:hgs} with the leading 
exponents $w\lambda-\rho(\bsk)\,\,(w\in W)$. 
Even if \eqref{eqn:condlambda} does not hold,
there are still $|W|$ 
linearly independent real analytic solutions on $\mathfrak{a}_+$ 
with leading exponents 
$w\lambda-\rho(\bsk)$ (counting the multiplicity on the wall in $\mathfrak{a}^*$). 
They may have polynomial terms (or logarithmic terms when taking $\exp x$ as a coordinate) 
as in the case of the Gauss hypergeometric differential equation. 
For general $\lambda\in\acs$, 
the asymptotic expansion \eqref{eqn:hcs2} becomes a convergent expansion on $\mathfrak{a}_+$ of 
the form
\begin{equation}\label{eqn:hcs3}
F(\lambda,\bsk;x)=\!\sum_{\mu\,\in\, W\lambda}\sum_{\kappa\,\in\, Q_+}\!\!p_{\mu,\kappa}(\lambda,\bsk;x)\,
e^{(\mu-\rho(\bsk)-\kappa)(x)},
\end{equation}
where $p_{\mu,\kappa}$  are polynomials in $x$. 
See \cite{HO87, Hec87}, \cite[\S 4.2]{Hec94}, and \cite{Os88, Os2005, Os2007}. 
\end{rem}

For $\bsk\in\mathcal{K}_{\C}$ let $\delta_\bsk$ denote the weight function on $\mathfrak{a}$ given by 
\begin{align*}
\delta_{\bsk} & =\!\prod_{\alpha\,\in\, \Rp}\!\big|e^{\frac12\alpha}-e^{-\frac12\alpha}\big|^{2\bsk_\alpha} \\
& =\!\prod_{\alpha\,\in\,\Rp_s}\!\big|e^{\frac12\alpha}-e^{-\frac12\alpha}\big|^{2\bsk_s+2\bsk_\l}
\big(e^{\frac12\alpha}+e^{-\frac12\alpha}\big)^{2\bsk_\l}
\!\!\prod_{\beta\,\in\,\Rp_m}\!\big|e^{\frac12\beta}-e^{-\frac12\beta}\big|^{2\bsk_m}.
\end{align*}
In order to describe a symmetry of $F(\lambda,\bsk)$ with respect to $\bsk$,
we introduce the multiplicity function $\tilde\bsk$ associated with $\bsk$, 
which is defined by 
\[
\tilde{\bsk}_s=\bsk_s+2\bsk_\l-1,\quad 
\tilde{\bsk}_m=\bsk_m,\quad
\tilde{\bsk}_\l=1-\bsk_\l.
\]
By \cite[Theorem 2.1.1]{Hec94}, 
\begin{align*}
\delta_\bsk^{\frac12}\circ(L(\bsk)+  \langle\rho(\bsk),&\rho(\bsk)\rangle)\circ\delta_\bsk^{-\frac12} \\
& =\delta_{\tilde{\bsk}}^{\frac12}\circ(L({\tilde{\bsk}})+  \langle\rho({\tilde{\bsk}}),\rho({\tilde{\bsk}})\rangle)\circ\delta_{\tilde{\bsk}}^{-\frac12}.
\end{align*}
It follows that
\begin{equation}\label{eq:ksymm0}
\varPhi(\lambda,\bsk)=\delta_{\tilde\bsk}^{\frac12} \delta_\bsk^{-\frac12} \varPhi(\lambda,\tilde\bsk).
\end{equation}
On the other hand,
in view of the characterization of the hypergeometric functions we have
\begin{equation}\label{eq:ksymm}
F(\lambda,\bsk)=2^{2\bsk_\l-1}\delta_{\tilde\bsk}^{\frac12} \delta_\bsk^{-\frac12} 
F(\lambda,\tilde{\bsk}).
\end{equation}
Here
\[
2^{2\bsk_\l-1}\delta_{\tilde\bsk}^{\frac12} \delta_\bsk^{-\frac12} 
=\prod_{i=1}^{r}\Big(\cosh\frac{\beta_i}{2}\Big)^{1-2\bsk_\l}
\]
is a nowhere vanishing analytic function on $\mathfrak a$.
For the elementary spherical functions associated with a one-dimensional $K$-type, 
the above formula was given by \cite[Theorem 5.2.2]{Hec94} and \cite[Proposition 2.6, Remark 3.8]{Sh94}. 
From \eqref{eqn:tcf} and \eqref{eqn:crho} we have
$\tilde\cc(\lambda,\bsk)=2^{2\bsk_\l-1}\tilde\cc(\lambda,\tilde\bsk)$
and $\tilde\cc(\rho(\bsk),\bsk)=\tilde\cc(\rho(\tilde\bsk),\tilde\bsk)$.
It follows that
\begin{equation}\label{eq:ksymm1}
\cc(\lambda,\bsk)=2^{2\bsk_\l-1}\cc(\lambda,\tilde\bsk).
\end{equation}

\section{Hypergeometric Fourier transform}\label{sect:ft}

Let $dx$ and $d\lambda$ denote the Lebesgue measures on the Euclidean spaces $\mathfrak{a}$ and 
$\mathfrak{a}^*$, respectively. 
Then 
\begin{equation}\label{eqn:measas}
d\lambda=
d\lambda_1d\lambda_2\cdots d\lambda_r, 
\end{equation}
in terms of the coordinates \eqref{eqn:coordas}. 
(Recall that we assume $||\beta_1||=2$. 
We should put the factor $({||\beta_1||}/{2})^r$ on the right hand side of \eqref{eqn:measas}
if we consider a general inner product on $\mathfrak{a}^*$.) 
For any $\eta\in\mathfrak a^*$
let $d\mu(\lambda)$ denote the measure on $\eta+\sqrt{-1}\mathfrak{a}^*$ 
given by
\begin{equation}\label{eqn:measure1}
d\mu(\lambda)=(2\pi)^{-r} d(\mathrm{Im}\,\lambda)=(2\pi\sqrt{-1})^{-r}d\lambda.
\end{equation}
If $\eta=0$ then the measures $dx$ on $\mathfrak{a}$ and $d\mu(\lambda)$ on $\sqrt{-1}\mathfrak{a}^*$ are 
normalized so that the inversion formula for the Euclidean Fourier transform
\[
\tilde{f}(\lambda)=\int_{\mathfrak{a}} f(x)\,e^{-\lambda (x)} dx\quad (\,f\in C_0^\infty(\mathfrak{a})\,)
\]
is given by
\[
f(x)=\int_{\sqrt{-1}\mathfrak{a}^*}\tilde{f}(\lambda)\,e^{\lambda(x)}d\mu(\lambda).
\]

A necessary and sufficient condition for the local integrability of the weight function $\delta_\bsk$ is given 
by \cite[Section 2]{BHO} and \cite[Proposition~5.1]{Hec97}. 
Clearly
$\delta_\bsk$ is locally integrable if and only if
$\delta_{\operatorname{Re}\bsk}$ is.
For a real $\bsk\in\mathcal{K}$, 
$\delta_\bsk$
is locally integrable if and only if 
\begin{equation}\label{eqn:mcond1}
\bsk_s+\bsk_\l>-\frac12+\max\{0,-(r-1)\bsk_m\}\text{ \,\,and\,\, } \bsk_m>-\frac1r.
\end{equation}
(Since $e^t-e^{-t}=2t+o(t)$, $\delta_\bsk$ is locally integrable if and only if the 
Selberg integral $S_r(\bsk_s+\bsk_\l+\frac12,1,\bsk_m)$ converges (cf. \cite{FW}).)  
Moreover, by \eqref{eqn:crho},  \eqref{eqn:mcond1} holds if and only if
$\bsk$ is in the connected component of $\mathcal{K}_\text{reg}\cap\mathcal{K}$ 
containing $\mathcal{K}_+$. 
Let $\mathcal K_1$ denote the set of
all $\bsk \in \mathcal K$ that satisfy $\eqref{eqn:mcond1}$.
If $\bsk\in\mathcal K_1$ then
the problem of giving the spectral decomposition 
of $L^2(\mathfrak{a}\,; \frac{1}{|W|}\delta_\bsk(x)dx)$
with respect to the hypergeometric function 
$F(\lambda,\bsk)$ makes sense.

Now for $\bsk\in\mathcal{K}_1+\sqrt{-1}\mathcal K$ we define  
the  hypergeometric Fourier transform $\mathcal{F}_\bsk$ of $f\in C^\infty_0(\mathfrak{a})^W$ by 
\begin{align}\label{eqn:hgfo}
\mathcal{F}_\bsk f(\lambda)&=\int_{\mathfrak{a_+}}\!f(x)F(\lambda,\bsk;x)\delta_{\bsk}(x)dx \\
& =\frac{1}{|W|}\int_{\mathfrak{a}}f(x)F(\lambda,\bsk;x)\delta_{\bsk}(x)dx, \notag
\end{align}
which is a holomorphic function in $\lambda\in\acs$.
(Note 
$F(\lambda,\bsk;-x)=F(-\lambda,\bsk;x)=F(\lambda,\bsk;x)$
since $-\operatorname{id}_{\mathfrak a}\in W$.)
Observe that $\mathcal{F}_\bsk f(\lambda)$ is also holomorphic in $\bsk$.
Given a $W$-invariant non-empty convex compact subset $C\subset \mathfrak{a}$,
consider the function $H_C(\lambda)=\max_{x\in C}\lambda (x)$ of $\lambda\in\mathfrak{a}^*$. 
\begin{prop}\label{prop:uniest}
Suppose the support of $f\in C_0^\infty(\mathfrak{a})^W$ is contained in $C$.
Then for any compact subset $U\subset \mathcal{K}_1+\sqrt{-1}\mathcal K$
and any $n\in\mathbb N$ 
\begin{equation}\label{eqn:pw}
\sup_{\substack{\lambda\in\acs,\,\bsk\in U}}(1+||\lambda||)^n e^{-H_C(\mathrm{Re}\,\lambda)}|\mathcal{F}_\bsk f(\lambda)|<+\infty.
\end{equation}
\end{prop}
For the proof of the proposition we need some preparation.
First, by the same method as in the proof of \cite[Proposition~6.1]{Op:Cherednik} we can prove
\begin{lem}\label{lem:Fest}
Suppose $\bsk\in \mathcal K_++\sqrt{-1}\mathcal K$
and put $\kappa=\frac12\sum_{\alpha\in \R^+}|\operatorname{Im}(\bsk_\alpha)|\alpha$.
Then for any $x\in \mathfrak a$ and $\lambda\in\acs$ it holds that
\[
|F(\lambda,\bsk;x)|\le |W|^{\frac12}e^{\max_{w\in W} \mathrm{Re}\,w\lambda(x)+\max_{w\in W} w\kappa(x)}.
\]
\end{lem}

Next, for $\bsk\in\mathcal K_\C$ and $\xi\in\mathfrak a$ we define the Cherednik operator
\[
T(\bsk,\xi)=\partial_\xi+\sum_{\alpha\in\R^+}\frac{\bsk_\alpha\alpha(\xi)}{1-e^{-\alpha}}(1-r_\alpha)-\rho(\bsk)(\xi)
\]
where $r_\alpha\in W$ is the reflection corresponding to $\alpha$.
This acts on $C^\infty(\mathfrak a)$ as a differential-difference operator.
If $f\in C_0^\infty(\mathfrak{a})$ is supported in $C$ then $T(\bsk,\xi)f$ is also.
The operators $T(\bsk,\xi)$ ($\xi\in\mathfrak a$) mutually commute and
define a unique action $T(\bsk,p)$ on $C^\infty(\mathfrak a)$ for any $p\in S(\mathfrak a_\C)$.
If $p\in S(\mathfrak a_\C)^W$ then $T(\bsk,p)$ commutes with the $W$-action
and there exists a differential operator $D(\bsk,p)\in\mathbb D(\bsk)$
with $\gamma_\bsk(D(\bsk,p))=p$
such that $T(\bsk,p)f=D(\bsk,p)f$ for any $f\in C^\infty(\mathfrak{a})^W$.
In particular
\begin{equation}\label{eq:CFlF}
T(\bsk,p)F(\lambda,\bsk;x)=p(\lambda)F(\lambda,\bsk;x)\quad\text{for any }p\in S(\mathfrak a_\C)^W.
\end{equation}

Let $\R^0=\R_\l\cup\R_m$ and put
\begin{align*}
\varDelta&=\prod_{\alpha\in \R^0\cap\R^+}(e^{\frac\alpha2}-e^{-\frac\alpha2})\in C^\infty(\mathfrak a),\\
\pi^\pm(\bsk)&=\prod_{\alpha\in\R_\ell^+}\Bigl(\alpha^\vee\pm\Bigl(\bsk_\l+\frac{\bsk_s}2\Bigr)\Bigr) 
\prod_{\alpha\in\R_m^+}(\alpha^\vee\pm\bsk_m) \in S(\mathfrak a_\C),\\
G^+(\bsk)&=\varDelta(x)^{-1}\circ T(\bsk,\pi^+),\\
G^-(\bsk)&=T(\bsk-\boldsymbol 1,\pi^-)\circ \varDelta(x).
\end{align*}
Here $\boldsymbol 1$ denotes the multiplicity function
which takes $1$ on $\R^0$ and $0$ on $\R_s$.
$G^\pm(\bsk)$ act on $C^\infty(\mathfrak a)^W$
and there exist $W$-invariant differential operators $\tilde G^\pm(\bsk)$ on $\mathfrak a$
such that $G^\pm(\bsk)f=\tilde G^\pm(\bsk)f$ for $f\in C^\infty(\mathfrak{a})^W$.
$\tilde G^\pm(\bsk)$
are called the hypergeometric shift operators with shift $\pm\boldsymbol 1$ (cf.~\cite[\S8.4.3]{HO2020}).
All coefficients of $\tilde G^\pm(\bsk)$
are analytic at least on $\mathfrak a_+$.
It is known that $\mathcal K_{\mathrm{reg}}+\boldsymbol 1\subset \mathcal K_{\mathrm{reg}}$
and that
\begin{equation}\label{eq:Fshift}
G^-(\bsk+\boldsymbol 1)F(\lambda,\bsk+\boldsymbol 1;x)=\frac{\tilde{\cc}(\rho(\bsk),\bsk)}{\tilde{\cc}(\rho(\bsk+\boldsymbol 1),\bsk+\boldsymbol 1)}F(\lambda,\bsk;x)
\end{equation}
(cf.~\cite[Corollaries 3.6.5 and 3.6.7]{Hec94}).
Also
\begin{equation}\label{eq:Phishift}
\tilde G^+(\bsk)\varPhi(\lambda,\bsk;x)
=(-1)^r\frac{\tilde{\cc}(-\lambda,\bsk)}{\tilde{\cc}(-\lambda,\bsk+\boldsymbol 1)}\varPhi(\lambda,\bsk+\boldsymbol 1;x)
\end{equation}
(cf.~\cite[Ch.~3]{Hec94}).
Here
\[
(-1)^r\frac{\tilde{\cc}(-\lambda,\bsk)}{\tilde{\cc}(-\lambda,\bsk+\boldsymbol 1)}
=(-1)^r\pi^+(\bsk)(-\lambda)=\pi^-(\bsk)(\lambda)
\]
is a polynomial.

\begin{lem}\label{lem:adjunction}
Suppose $\bsk\in\mathcal K_1+\sqrt{-1}\mathcal K$, $f\in C^\infty_0(\mathfrak a)^W$
and $g\in C^\infty(\mathfrak a)^W$.
Then
\[
\int_{\mathfrak a} f(x) (G^-(\bsk+\boldsymbol 1)g(x))\delta_\bsk(x)dx
=
(-1)^r
\int_{\mathfrak a} (G^+(\bsk)f(x))g(x) \delta_{\bsk+\boldsymbol 1}(x)dx.
\]
\end{lem}
\begin{proof}
We use the graded Hecke algebra $\boldsymbol H_\bsk$,
which is the unique unital $\C$-algebra with the following properties:\\
\noindent(1) \,
$\boldsymbol H_\bsk$ contains $S(\mathfrak a_\C)$ and the group algebra $\C W$ as subalgebras;\\
\noindent(2) \,
the multiplication map $S(\mathfrak a_\C)\otimes\C W \to \boldsymbol H_\bsk$ is a linear isomorphism;\\
\noindent(3) \,
$r_\alpha \cdot \xi = r_\alpha(\xi) \cdot r_\alpha -(\bsk_\alpha+2\bsk_{2\alpha})\alpha(\xi)$
for $\xi\in\mathfrak a$ and $\alpha\in\B$.\\
\noindent
Importantly the $S(\mathfrak a_\C)$-action $T(\bsk,\cdot)$
and the usual $W$-action on $C^\infty(\mathfrak a_\C)$ are integrated into
an $\boldsymbol H_\bsk$-action on $C^\infty(\mathfrak a_\C)$.
Put 
\begin{equation*}
\epsilon^+=|W|^{-1}\sum_{w\in W}w,\quad
\epsilon^-=|W|^{-1}\sum_{w\in W}(\sgn w)w\,
 \in \C W.
\end{equation*}
We assert that
\begin{equation}\label{eq:epepe}
\epsilon^-\pi^-\epsilon^+=\pi^+\epsilon^+\quad\text{in }\boldsymbol H_\bsk.
\end{equation}
Indeed, it follows from Property (2) of $\boldsymbol H_\bsk$ that
the map $S(\mathfrak a_\C)\to \boldsymbol H_\bsk\epsilon^+$ defined by $p\mapsto p\epsilon^+$
is a linear isomorphism. 
Let $V$ be the subspace of $S(\mathfrak a_\C)$
consisting of the elements of degree $\le|\R^0\cap\R^+|$.
Then $V\epsilon^+$ is a left $W$-module and
it easily follows from the theory of $W$-harmonic polynomials
that a skew element in $V\epsilon^-$ is unique up to a scalar multiple. 
On the other hand, by Property (3),
both the sides of \eqref{eq:epepe} are skew and
have the same top degree term $\prod_{\alpha\in \R^0\cap\R^+}\alpha^\vee\epsilon^+$.
Thus the assertion is proved.

Let $w^*=-\operatorname{id}_{\mathfrak a}$ be the longest element of $W$.
Then $\sgn(w^*)=(-1)^r$.
By \cite[Lemma 7.8]{Op:Cherednik}
it holds that 
\begin{equation}\label{eq:adjunction}
\int_{\mathfrak a} \varphi(x) (T(\bsk,p)\psi(x))\delta_\bsk(x)dx\\
=
\int_{\mathfrak a} (w^* T(\bsk,p) w^*\varphi(x)) \psi(x)\delta_\bsk(x)dx
\end{equation}
for $p\in S(\mathfrak a_\C)$, $\varphi\in C^\infty_0(\mathfrak a)$ and $\psi\in C^\infty(\mathfrak a)$.
Now we calculate
\begin{align*}
\int_{\mathfrak a} &f(x) (T(\bsk,\pi^-)\varDelta(x) g(x))\delta_\bsk(x)dx\\
&=\int_{\mathfrak a} (w^*T(\bsk,\pi^-)f(x)) \varDelta(x) g(x)\delta_\bsk(x)dx\\
&=(-1)^r \int_{\mathfrak a} (\epsilon^-T(\bsk,\pi^-)\epsilon^+f(x)) \varDelta(x) g(x)\delta_\bsk(x)dx
&&(\,\text{as } \varDelta(x)\text{ is skew}\,)\\
&=(-1)^r \int_{\mathfrak a} (T(\bsk,\pi^+)f(x)) \varDelta(x) g(x)\delta_\bsk(x)dx
&&(\,\text{by } \eqref{eq:epepe}\,)\\
&=(-1)^r \int_{\mathfrak a} (G^+(\bsk)f(x)) g(x) \varDelta(x)^2 \delta_\bsk(x)dx\\
&=(-1)^r \int_{\mathfrak a} (G^+(\bsk)f(x)) g(x) \delta_{\bsk+\boldsymbol 1}(x)dx.&&
\qedhere
\end{align*}
\end{proof}

\begin{proof}[Proof of Proposition \ref{prop:uniest}]
Let $\bsk\in \mathcal K_1+\sqrt{-1}\mathcal K$.
In view of \eqref{eqn:mcond1} we have
$\tilde\bsk, \bsk+\boldsymbol 1\in \mathcal K_1+\sqrt{-1}\mathcal K$.
By \eqref{eq:ksymm}
\begin{equation}\label{eq:HGFshift1}
\mathcal F_\bsk f=\mathcal F_{\tilde\bsk}
\biggl(4^{2\bsk_\l-1}\prod_{i=1}^{r}\Big(\cosh\frac{\beta_i}{2}\Big)^{2\bsk_\l-1}f\biggr).
\end{equation}
Also, by \eqref{eq:Fshift} and Lemma~\ref{lem:adjunction}
\begin{equation}\label{eq:HGFshift2}
\mathcal F_\bsk f=\mathcal F_{\bsk+\boldsymbol 1}
\biggl(
(-1)^r\frac{\tilde{\cc}(\rho(\bsk+\boldsymbol 1),\bsk+\boldsymbol 1)}{\tilde{\cc}(\rho(\bsk),\bsk)}
G^+(\bsk)f\biggr).
\end{equation}
Let $p\in S(\mathfrak a_\C)^W$.
Then by \eqref{eq:CFlF} and \eqref{eq:adjunction}
\begin{equation}\label{eq:dualop}
p(\lambda)\mathcal F_\bsk f(\lambda)
=\mathcal F_\bsk (T(\bsk,p)f)(\lambda).
\end{equation}

Now let us choose $a, b\in\mathbb N$ so that
$\bsk':=\widetilde{(\bsk+a\boldsymbol 1)}+b\boldsymbol 1\in \mathcal K_++\sqrt{-1}\mathcal K$
for all $\bsk\in U$.
By the three formulas above
there exists a $C^\infty$-function $g(\bsk,x)$ on $(\mathcal K_1+\sqrt{-1}\mathcal K)\times \mathfrak a$
which is $W$-invariant in $x$, supported in $(\mathcal K_1+\sqrt{-1}\mathcal K)\times C$,
and satisfying
$p(\lambda)\mathcal F_\bsk f(\lambda)=\mathcal F_{\bsk'}(g(\bsk,\cdot))(\lambda)$.
Hence the proposition follows from Lemma \ref{lem:Fest}
(cf.~the proof of \cite[Theorem 4.1]{Op99}).
\end{proof}

\begin{rem}
An argument similar to the above is used in \cite{HoO1} in the setting 
of the Jacobi polynomial associated with the root system of type $BC$.
\end{rem}

Let $\PW$ denote the space of all holomorphic functions $\phi$ on $\acs$ such that
\begin{equation}\label{eq:PWcondition}
\sup_{\substack{\lambda\in\acs}}(1+||\lambda||)^n e^{-H_C(\mathrm{Re}\,\lambda)}|\phi(\lambda)|<+\infty\quad
(\forall n\in\mathbb N)
\end{equation}
for some $W$-invariant non-empty convex compact subset $C\subset\mathfrak a$.
 It coincides with the Paley-Wiener space in the Euclidean Fourier analysis 
(the case of $\bsk=0$). 
By Proposition~\ref{prop:uniest}, $\mathcal F_\bsk f\in \PW^W$
for $f\in C_0^\infty(\mathfrak{a})^W$.

Now suppose $\bsk\in\mathcal K_1+\sqrt{-1}\mathcal K$.
By \eqref{eqn:tcf} we can choose 
$\eta\in -\Cl(\mathfrak{a}_+^*)$ so that  
$\cc(-\lambda,\bsk)^{-1}$ is regular on
$\{\lambda\in\acs\,;\,\text{Re}\,\lambda\in \eta-\Cl(\mathfrak{a}_+^*)\}$. 
For $\phi\in \PW^W$ define
\begin{equation}
 \mathcal{J}_\bsk \,\phi(x)=\int_{\eta+\sqrt{-1}\mathfrak{a}^*}\phi(\lambda)\,\varPhi(\lambda,\bsk;x)
\,\cc(-\lambda,\bsk)^{-1}\,d\mu(\lambda)
\quad (x\in\mathfrak{a}_+). 
\label{eqn:jf}
\end{equation}
If $T\subset\mathfrak a^*$ is a small compact neighborhood of $\eta$ then
the polynomial $p(\lambda)$ in Lemma~\ref{lem:esthcs} is $1$.
Hence by Lemma \ref{lem:esthcs}, Lemma~\ref{lem:estcf} and \eqref{eq:PWcondition},
the integral on the right hand side of \eqref{eqn:jf} does not depend on the choice of $\eta$ and converges to
a $C^\infty$-function on $\mathfrak a_+$.
Moreover $\mathcal{J}_\bsk \,\phi(x)$ is holomorphic in $\bsk$ for each fixed $x$.
By \eqref{eq:ksymm0} and \eqref{eq:ksymm1}
\begin{equation}\label{eq:WPOshift1}
4^{2\bsk_\l-1}\prod_{i=1}^{r}\Big(\cosh\frac{\beta_i}{2}\Big)^{2\bsk_\l-1}\mathcal{J}_\bsk \,\phi(x)
=\mathcal{J}_{\tilde\bsk} \,\phi(x).
\end{equation}
Also, by \eqref{eq:Phishift}
\begin{equation}\label{eq:WPOshift2}
(-1)^r\frac{\tilde{\cc}(\rho(\bsk+\boldsymbol 1),\bsk+\boldsymbol 1)}{\tilde{\cc}(\rho(\bsk),\bsk)}
\tilde G^+(\bsk)\mathcal{J}_\bsk \,\phi(x)
=\mathcal{J}_{\bsk+\boldsymbol 1} \,\phi(x).
\end{equation}

\begin{thm}[Inversion formula, first form] \label{thm:1st}
Suppose $\bsk\in\mathcal K_1+\sqrt{-1}\mathcal K$
and choose $\eta$ as above. Suppose $f\in C^\infty_0(\mathfrak{a})^W$ and  $x\in\mathfrak{a}_+$.
Then $f(x)=\mathcal{J}_\bsk\,\mathcal{F}_\bsk f(x)$,
namely
\begin{equation}\label{eqn:1st}
f(x)=
\int_{\eta+\sqrt{-1}\mathfrak{a}^*}\!\mathcal{F}_\bsk 
f(\lambda)\,\varPhi(\lambda,\bsk;x)\,\cc(-\lambda,\bsk)^{-1}
d\mu(\lambda).
\end{equation}
\end{thm}

\begin{proof}
Lemma~\ref{lem:esthcs}, Lemma~\ref{lem:estcf}, and Proposition~\ref{prop:uniest} impliy that  
the integral on the right hand side of \eqref{eqn:1st} is holomorphic in $\bsk$. 
By \cite[Theorem~9.13]{Op:Cherednik}
\eqref{eqn:1st}
holds for $\bsk\in\mathcal{K}_+$. 
The general case follows by analytic continuation. 
\end{proof}

\begin{prop}\label{prop:J_kbdd}
For any $\phi\in\PW^W$ the support of $\mathcal J_\bsk\,\phi$ is bounded.
\end{prop}
\begin{proof}
Take $y\in\mathfrak a_+$ so that 
$\phi$ satisfies \eqref{eq:PWcondition}
with $C$ being the convex hull of $W y$.
If $\bsk\in\mathcal K_+$ then 
$\mathcal J_\bsk\,\phi(x)=0$ for any $x\in\mathfrak a_+\,\setminus\, C$
by \cite[Theorem 8.6]{Op:Cherednik}. 
This holds for all $\bsk\in\mathcal K_1+\sqrt{-1}\mathcal K$
by analytic continuation. 
\end{proof}

Now suppose $\bsk\in\mathcal K_1$.
Since $2\bsk_s+2\bsk_\l+1>0$ by \eqref{eqn:mcond1},
$(\bsk_s+1,\bsk_s+2\bsk_\l)\notin (-\mathbb N)^2$.
Hence from \eqref{eqn:tcf} we see that
$\cc(-\lambda,\bsk)^{-1}$ is regular on $\sqrt{-1}\mathfrak a^*$
as a function in $\lambda$.
For $\phi\in\PW^W$ and $x\in\mathfrak{a}_+$ we define
\begin{align}
\mathcal{J}_{\bsk,\emptyset}\,\phi(x)&=\int_{\sqrt{-1}\mathfrak{a}^*}\phi(\lambda)\,\varPhi(\lambda,\bsk;x)\,
\cc(-\lambda,\bsk)^{-1}\,d\mu(\lambda)\label{eqn:jfempty}\\
&=\frac{1}{|W|}
\int_{\sqrt{-1}\mathfrak{a}^*}\phi(\lambda)\,F(\lambda,\bsk;x)\,|\cc(\lambda,\bsk)|^{-2}d\mu(\lambda).\notag
\end{align}
The right hand side of the first line converges to
a $C^\infty$-function on $\mathfrak a_+$ by Lemma \ref{lem:esthcs}
and Lemma \ref{lem:estcf}
(with $a_j(\bsk)$ and $b_j(\bsk)$ being constant functions)
and becomes the second line by changes of variables and \eqref{eqn:hcs2}.
The meaning of the symbol $\mathcal{J}_{\bsk,\emptyset}$ will be clear in Section \ref{sect:inversion}.
$\mathcal{J}_{\bsk,\emptyset}\,\phi$ extends to a $W$-invariant $C^\infty$-function on $\mathfrak a$ by the following lemma.

\begin{lem}\label{lem:est}
Let $\bsk\in\mathcal{K}_\mathrm{reg}\cap\mathcal K$. Let $V\subset\mathfrak{a}$ be a compact set 
and $p\in S(\mathfrak{a}_\C)$. Then there exist constants $C>0$ and $n\in\mathbb{N}$ such that 
\[
|p(\partial_x)F(\lambda,\bsk;x)|\leq C(1+||\lambda||)^n e^{\max_{w\in W} \mathrm{Re}\,w\lambda(x)}
\]
for all $\lambda\in\acs$ and $x\in V$.
\end{lem}
\begin{proof}
This is a version of \cite[Theorem~2.5]{Op99} for the type $BC_r$ root system
and the same proof can apply.
If $\bsk\in \mathcal{K}_+$ then the estimate is true
by \cite[Corollary~6.2]{Op:Cherednik}.
From this we can deduce the general case 
using \eqref{eq:ksymm} and \eqref{eq:Fshift}. 
\end{proof}

\begin{rem}
Ho and \'Olafsson \cite[Appendix]{HoO2} proved an estimate as in the above lemma for 
$\bsk\in\mathcal{K}$ with $\bsk_s+\bsk_\l\geq 0,\,\bsk_s\geq 0$, and $\bsk_m\geq 0$ by generalizing 
the proof of \cite[Corollary~6.2]{Op:Cherednik}.
\end{rem}

\begin{thm}\label{thm:cont}
Suppose $\bsk\in\mathcal{K}_1$  satisfies
\begin{equation}\label{eqn:mcond3}
\bsk_s\geq -1,\quad\bsk_m\geq 0,\quad\bsk_s+2\bsk_\l\geq 0.
\end{equation}
Then for 
$f\in C_0^\infty(\mathfrak{a})^W$ we have $f=\mathcal{J}_{\bsk,\emptyset}\,\mathcal{F}_\bsk f$, namely
\begin{equation}\label{eqn:invmc}
f(x)=\frac{1}{|W|}
\int_{\sqrt{-1}\mathfrak{a}^*}\mathcal{F}_\bsk
f(\lambda)\,F(\lambda,\bsk;x)\,|\cc(\lambda,\bsk)|^{-2}d\mu(\lambda)\quad 
(x\in\mathfrak{a}).
\end{equation}
\end{thm}
\begin{proof}
If $\bsk\in\mathcal{K}$ satisfies \eqref{eqn:mcond3}, then 
we can choose $\eta=0$ in Theorem \ref{thm:1st} (cf.~\eqref{eqn:tcf}).
\end{proof}

\section{Tempered hypergeometric functions}\label{sect:tempered}

In this section, we define the notion of tempered hypergeometric functions by 
the growth condition \eqref{eqn:tempered} and give a sufficient condtion for 
%a  hypergeometric function to be tempered 
temperedness 
 (Corollary~\ref{cor:tempered2}). 
We will see in Section~\ref{sect:inversion} that those tempered hypergeometric 
functions contribute to the inversion formula and the Plancherel theorem for 
the hypergeometric Fourier transform. 
%tempered hypergeometric functions that contribute to the Plancherel theorem, 
%which we shall study in Section~\ref{sect:inversion}. 
Though combinatorial features are different, our argument and result are similar to those of \cite{Op99}, 
where tempered hypergeometric functions are studied for some negative multiplicity functions on 
 reduced root systems. 

If $\bsk\in \mathcal{K}_1$ satisfies \eqref{eqn:mcond3}, then there are only continuous spectra in 
the inversion formula \eqref{eqn:invmc} for the hypergeometric Fourier transform. 
In this case, only the tempered hypergeometric functions $F(\lambda,\bsk)$ with $\lambda\in\sqrt{-1}\mathfrak{a}^*$ contribute.
If $\bsk\in\mathcal{K}_1$ does not satisfy \eqref{eqn:mcond3}, then 
$\cc(-\lambda,\bsk)^{-1}$ has non-negligible singularities and
we must take account of residues
to shift the domain of integration
in the right hand side of \eqref{eqn:1st} as $\eta\to 0$. 
The most continuous part of the spectral decomposition is given by the right hand side of 
\eqref{eqn:invmc}. Besides it, there appear spectra whose supports have dimensions lower than $r$. 

Recall $\mathcal K_1$ consists of all $\bsk\in \mathcal K$ satisfying \eqref{eqn:mcond1}.
In this paper, we exclude the case of $\bsk_m< 0$ in \eqref{eqn:mcond1} and study the case of
\begin{equation}\label{eqn:mcond2}
\bsk_s+\bsk_\l>-\frac12,\quad  \bsk_m\geq 0.
\end{equation}
Under \eqref{eqn:mcond2} the residue calculus can be handled explicitly by 
the same method 
as in the proof of \cite[Theorem~6.7]{Sh94}, where the inversion formula 
for the spherical transform associated with 
 a one-dimensional $K$-type on a simple Lie group of Hermitian type 
is given. 
The case of 
\begin{equation}\label{eqn:mcond20}
\bsk_s+\bsk_\l>-\frac12,\quad  \bsk_m= 0
\end{equation}
reduces to the case of $BC_1$ and is easy to analyze. 
If $\bsk_m<0$, then we might need a different kind of  combinatorial 
argument as in \cite{Op99} and
 we do not go further into this case in this paper.
Let us denote the set of
all $\bsk \in \mathcal K$ that satisfy $\eqref{eqn:mcond2}$ by $\mathcal K'_1$.

It is convenient to adopt the parameters $\al$ and $\be$ for the Jacobi function (cf. \cite{FJ77, Ko75, Ko84}) as in the case of $r=1$. 
 We
substitute 
\begin{equation}\label{eqn:paraj}
\al=\bsk_s+\bsk_\l-\frac12,\quad\be=\bsk_\l-\frac12
\end{equation}
for the multiplicity parameters $\bsk_s$ and $\bsk_\l$. 
Then
$\bsk\in\mathcal K$ belongs to $\mathcal K'_1$ 
if and only if
\begin{equation}\label{eqn:mcond2b}
\al>-1,\quad \bsk_m\geq 0
\end{equation}
and $\rho(\bsk)$ is written as
\begin{equation}
\rho(\bsk)=(\al+\be+1,\al+\be+2\bsk_m+1,\dots,\al+\be+2(r-1)\bsk_m+1).
\end{equation}

For $1\leq i\leq r$ let $\tilde{\cc}_i(\lambda,\bsk)$ denote  the product of the factors of \eqref{eqn:tcf} for 
 the roots $\beta_i$ and $2\beta_i$. 
That is,
\begin{align}\label{eqn:tcfbc1}
\tilde{\cc}_i(\lambda,\bsk)&:=\tilde{\cc}_{\beta_i}(\lambda,\bsk)\tilde{\cc}_{2\beta_i}(\lambda,\bsk) \\
&=
\frac{2^{-\al+\be}\varGamma\big(\frac12\lambda_i\big)\varGamma\big(\frac12(\lambda_i+1)\big)}
{\varGamma\big(\frac12(\lambda_i+\al-\be+1)\big)\varGamma\big(\frac12(\lambda_i+\al+\be+1)\big)}
\notag\\
&=
\frac{2^{-\al+\be-\lambda_i+1}\sqrt{\pi}\varGamma(\lambda_i)}
{\varGamma\big(\frac12(\lambda_i+\al-|\be|+1)\big)\varGamma\big(\frac12(\lambda_i+\al+|\be|+1)\big)}.
\notag
\end{align}

Let $\varTheta$ be a subset of $\B$, 
$\langle\varTheta\rangle$ the subset of $\R$ consisting of  the linear combinations of 
elements in $\varTheta$, and
 $W_\varTheta$ the subgroups of $W$ generated by 
$\{r_\alpha\,;\,\alpha\in\varTheta\}$.
 Define 
\begin{align*}
& \mathfrak{a}_\varTheta=\{x\in  \mathfrak{a}\,;\,\alpha(x)=0\,\,\text{ for all }\alpha\in \varTheta\}, \\
& \mathfrak{a}(\varTheta)=\{x\in\mathfrak{a}\,;\,\langle x,y\rangle=0\,\,\text{ for all }y\in \mathfrak{a}_\varTheta\}.
\end{align*}
Then $\langle\varTheta\rangle$ is a root system in $\mathfrak{a}(\varTheta)^*$ with a positive system 
$\langle\varTheta\rangle^+=\langle\varTheta\rangle\cap \Rp$, the set $\varTheta$ of simple roots, 
and the Weyl group $W_\varTheta$. 
For $\lambda\in\mathfrak{a}_\mathbb{C}^*$ we write 
$\lambda=\lambda_{\mathfrak{a}(\varTheta)}+\lambda_{\mathfrak{a}_{\varTheta}}$ with 
$\lambda_{\mathfrak{a}(\varTheta)}\in \mathfrak{a}(\varTheta)^*_\mathbb{C}$ and 
$\lambda_{\mathfrak{a}_{\varTheta}}\in \mathfrak{a}_{\varTheta, \mathbb{C}}^*$. 
Thus $\rho(\bsk)_{\mathfrak{a}(\varTheta)}$ is ``$\rho(\bsk)$'' for the root system $\langle\varTheta\rangle$.
Put
\begin{equation*}
 \tilde{\cc}_\varTheta(\lambda,\bsk)=\!\prod_{\alpha\,\in\, \langle\varTheta\rangle^+}\!
\tilde{\cc}_\alpha(\lambda,\bsk),
\quad \tilde{\cc}^\varTheta(\lambda,\bsk)=\!\prod_{\alpha\,\in\, \Rp\,\setminus\, \langle\varTheta\rangle^+}\!
\tilde{\cc}_\alpha(\lambda,\bsk).
\end{equation*}
In particular $\tilde{\cc}_\emptyset(\lambda,\bsk)=1$. 
Moreover, define
\begin{equation}
 \cc_\varTheta(\lambda,\bsk)=\frac{\tilde{\cc}_\varTheta(\lambda,\bsk)}{\tilde{\cc}_\varTheta(\rho(\bsk),\bsk)}, 
\quad
 \cc^\varTheta(\lambda,\bsk)=\frac{\tilde{\cc}^\varTheta(\lambda,\bsk)}{\tilde{\cc}^\varTheta(\rho(\bsk),\bsk)}.
\end{equation}
Thus
\begin{equation}\label{eq:updaownTheta}
{\cc}(\lambda,\bsk)={\cc}_\varTheta(\lambda,\bsk)\,{\cc}^\varTheta(\lambda,\bsk)
\end{equation}
and 
${\cc}_\varTheta(\lambda,\bsk)$ is the $\cc$-function for the root system $\langle\varTheta\rangle$. 

For $0\leq i\leq r$ let $\varTheta_i$  denote the subset of $\B$ 
given by
\begin{equation}\label{eqn:thetai}
 \varTheta_i=\{\alpha_j\,;\,r-i+1\leq j\leq r\}.
\end{equation}
We have
\begin{align*}
& \mathfrak{a}(\varTheta_i)^*=\mathrm{Span}_\mathbb{R}\{\beta_j\,;\,1\leq j\leq i\}, \\
& \mathfrak{a}_{\varTheta_i}^*=\mathrm{Span}_\mathbb{R}\{ \beta_j\,;\,i+1\leq j\leq r\}. 
\end{align*}
Define the subset $W^{\varTheta_i}$  of $W$ by 
\begin{equation}\label{eq:W^Theta_i}
W^{\varTheta_i}=\{w\in W\,;\,w\langle\varTheta_i\rangle^+\subset \R^+\}.
\end{equation}
This is a complete set of representatives for $W/W_{\varTheta_i}$. 
For $1\leq i\leq r$, 
$\langle\varTheta_i\rangle\subset\mathfrak{a}(\varTheta_i)^*$ is a root system of type $BC_{i}$, 
its Weyl group is $W_{\varTheta_i}\simeq \mathbb{Z}_2^i\rtimes \mathfrak{S}_i$, 
and 
\begin{align}
W^{\varTheta_i}= \label{W^Theta_i1}
\{(\varepsilon,\sigma)\in \mathbb{Z}_2^r\rtimes \mathfrak{S}_r\,;\,
& \sigma(1)<\cdots<\sigma(i),\\
& \varepsilon(\beta_j)=\beta_j \,\,(\,\forall\,
j\in\{\sigma(1),\dots,\sigma(i)\}\,)\}. \notag
\end{align}

Hereafter in this section we assume $\bsk\in\mathcal K'_1$.
For $1\leq i\leq r$ let $\Dki$ denote the finite subset
 of $\mathfrak{a}(\varTheta_i)^*\simeq \mathbb{R}^i$ given  by 
\begin{align}\label{eqn:D}
\Dki=\{(\lambda_1,\dots,\lambda_i)\in\,&\mathbb{R}^i\,;\, \lambda_1+|\be|-\al-1\in 2\mathbb{N},\,\lambda_i<0,\, \\
& \lambda_{j+1}-\lambda_j-2\bsk_m\in 2\mathbb{N}\,\,(1\leq j\leq i-1)\}.\notag
\end{align}
Note 
$\Dki\not=\emptyset$  if and only if $\al-|\be|+2(i-1)\bsk_m+1<0$. 
In particular $\be\ne0$ if $\Dki\not=\emptyset$.
We set $\Dkz=\mathbb R^0=\{0\}$ for convenience.

\begin{thm}
\label{thm:temperedtheta1}
Suppose $0\le i \le r$,
$\lambda_{\mathfrak{a}(\varTheta_i)}\in \Dki$ and $x\in\mathfrak a_+$.
If we write $\lambda=\lambda_{\mathfrak{a}(\varTheta_i)}+\lambda_{\mathfrak{a}_{\varTheta_i}}$
for $\lambda_{\mathfrak{a}_{\varTheta_i}}\in \mathfrak a_{\varTheta_i,\C}^*$ 
then
\begin{equation}\label{eqn:exptheta2}
F(\lambda,\bsk;x)=\!\sum_{w\,\in\, W^{\varTheta_i}}
\!\!\cc(w\lambda,\bsk)\,\varPhi(w\lambda,\bsk;x)
\end{equation}
as a meromorphic function in $\lambda_{\mathfrak a_{\varTheta_i}}$.
In particular, if $\lambda\in\Dkr$, then 
\begin{equation}\label{eqn:fthetadeg3}
F(\lambda,\bsk;x)=\cc(\lambda,\bsk)\,\varPhi(\lambda,\bsk;x).
\end{equation}
\end{thm}
\begin{rem}\label{rem:regularization}
When $\langle \lambda_{\mathfrak{a}(\varTheta_i)},\alpha^\vee\rangle\in\mathbb Z$
for some $\alpha\in \langle\varTheta_i\rangle^+$
the right hand side of \eqref{eqn:exptheta2} should be interpreted as follows.
We assume $\be>0$.
(The argument in the case of $\be<0$ is similar.)
Then $\al-|\be|+1=\bsk_s+1$.
For $\Delta\bsk=(\Delta\bsk_s,\Delta\bsk_m,\Delta\bsk_\ell)$ in a 
sufficiently small neighborhood of $0\in \mathcal K_{\mathbb C}$
put
\[
\Delta\lambda_{\mathfrak{a}(\varTheta_i)}
=
(\Delta\bsk_s,\Delta\bsk_s+2\Delta\bsk_m,\ldots,\Delta\bsk_s+2(i-1)\Delta\bsk_m)
\in
\mathfrak{a}(\varTheta_i)_{\mathbb{C}}^*,
\]
so that $\lambda_{\mathfrak{a}(\varTheta_i)}+\Delta\lambda_{\mathfrak{a}(\varTheta_i)} \in D_{\bsk+\Delta\bsk}(\varTheta_i)$
when $\Delta\bsk$ is real.
We see from \eqref{eqn:D} that
$\langle\lambda+\Delta\lambda_{\mathfrak{a}(\varTheta_i)},\alpha^\vee\rangle\not\in \mathbb{Z}$
($\forall\alpha\in\R$) for generic $\Delta\bsk$ and $\lambda_{\mathfrak a_{\varTheta_i}}$.
Thus
\[
\sum_{w\in W^{\varTheta_i}}
\cc(w(\lambda+\Delta\lambda_{\mathfrak{a}(\varTheta_i)}),\bsk+\Delta\bsk)\,\varPhi(w(\lambda+\Delta\lambda_{\mathfrak{a}(\varTheta_i)}),\bsk+\Delta\bsk;x)
\]
is a meromorphic function in $\Delta\bsk$
and $\lambda_{\mathfrak a_{\varTheta_i}}$
and is equal to the holomorphic function $F(\lambda+\Delta\lambda_{\mathfrak{a}(\varTheta_i)},\bsk+\Delta\bsk;x)$
by the proof below.
The right hand side of \eqref{eqn:exptheta2} is the restriction of this function to $\Delta\bsk=0$.
\end{rem}
\begin{proof}[Proof of Theorem \ref{thm:temperedtheta1}]
First we assume that 
$\bsk$ and 
$\lambda_{\mathfrak{a}_{\varTheta_i}}\in\mathfrak{a}_{\varTheta_i,\C}^*$ are generic so that 
$\lambda=\lambda_{\mathfrak{a}(\varTheta_i)}+\lambda_{\mathfrak{a}_{\varTheta_i}}$  
satisfies \eqref{eqn:condlambda}. 
(We consider $\lambda_{\mathfrak{a}(\varTheta_i)}$ varies with $\bsk$ as in the remark above.)
We will show that all terms for $w\in W\,\setminus\,{W^{\varTheta_i}}$ 
in the right hand side of 
\eqref{eqn:hcs2} vanish. 
So suppose $w\in W\,\setminus\,{W^{\varTheta_i}}$.
By \eqref{eq:W^Theta_i} there exists $\alpha\in \varTheta_i$
such that $w\alpha\in -\R^+$.
If $\alpha=\alpha_r=\beta_1$ and $w\beta_1=-\beta_j$ with $1\le j \le r$,
then $\tilde{\cc}_j(w\lambda,\bsk)=0$ by 
\eqref{eqn:tcfbc1} and \eqref{eqn:D}.
If $\alpha=\alpha_{r-j}=\beta_{j+1}-\beta_j$ with $1\leq j\leq i-1$,
then by \eqref{W^Theta_i1}
$\cc(w\lambda,\bsk)$ 
 contains a factor
\[
\frac{\varGamma\big(\frac12(\lambda_j-\lambda_{j+1})\big)
}{\varGamma\big(\frac12(\lambda_j-\lambda_{j+1}+2\bsk_m)\big)},
\] 
which vanishes 
because $\frac12(\lambda_j-\lambda_{j+1}+2\bsk_m)$ 
in the denominator becomes a non-positive integer by \eqref{eqn:D}. 
Hence \eqref{eqn:exptheta2} holds.
We may drop our assumption on $\bsk$ and $\lambda_{\mathfrak{a}_{\varTheta_i}}$ by analytic continuation. 
\end{proof}

We say that $F(\lambda,\bsk)$ is tempered if there exist $C\geq 0$ and $d\in\mathbb{N}$ such that 
\begin{equation}\label{eqn:tempered}
|F(\lambda,\bsk;x)|\leq C(1+||x||)^d\, e^{-\rho(\bsk)(x)}
\quad\text{for all } \,
x\in\Cl(\mathfrak{a}_+).
\end{equation}

\begin{cor}\label{cor:tempered2}
Let $\bsk\in\mathcal{K}'_1$. 
For $\lambda
\in \Dki+\sqrt{-1}\mathfrak{a}_{\varTheta_i}^*$, 
 $F(\lambda,\bsk)$ is tempered. 
Moreover,  $F(\lambda,\bsk)$ 
is a real-valued square integrable function (with respect to $\delta_{\bsk}(x)dx$)
for  any $\lambda\in \Dkr$.
\end{cor}
\begin{proof}
By Theorem~\ref{thm:temperedtheta1}, we have a convergent expansion on $\mathfrak{a}_+$ of the form
\begin{equation}\label{eqn:hcs4}
F(\lambda,\bsk;x)=
\!\sum_{\mu\,\in\, W^{\varTheta_i}\lambda}\sum_{\kappa\,\in\, Q_+}\!p_{\mu,\kappa}(\lambda,\bsk;x)\,
e^{(\mu-\rho(\bsk)-\kappa)(x)},
\end{equation}
where $p_{\mu,\kappa}$  are polynomials in $x$. 
The leading exponents of $F(\lambda,\bsk;x)$ for $\lambda\in \Dki+\sqrt{-1}\mathfrak{a}_{\varTheta_i}^*$ 
on $\mathfrak{a}_+$  are $w\lambda-\rho(\bsk)\,\,(w\in W^{\varTheta_i})$. 
For $w=(\varepsilon,\sigma)\in W^{\varTheta_i}$ we have
\begin{align*}
& \langle \mathrm{Re}\,
w\lambda,\beta_j\rangle=\langle w\lambda_{\mathfrak{a}(\varTheta_i)},\beta_j\rangle<0\,\,\text{ for any }\,
j\in\{\sigma(1),\dots,\sigma(i)\}, \\
& \langle \mathrm{Re}\,w\lambda,\beta_j\rangle=0\,\,\text{ for any }\,
j\not\in\{\sigma(1),\dots,\sigma(i)\}.
\end{align*}
Our results follows from the criterion of Casselman and Mili\v{c}i\'c (\cite[Corollary~7.2, Theorem~7.5]{CM}.
For $\lambda\in\Dkr$, $F(\lambda,\bsk)$ is real-valued by \eqref{eqn:fthetadeg3}. 
\end{proof}

In Section~\ref{sect:inversion} we shall show that $F(\lambda,\bsk)\,\,(\lambda\in\Dkr)$
 exhaust the square integrable hypergeometric 
functions after establishing the Plancherel theorem (see Corollary~\ref{cor:lthg}). 

Now let $\varPi_\bsk$ denote the set of real affine subspaces of $\acs$ of the form
\[
L=w(\lambda_{\mathfrak{a}(\varTheta_i)}+\sqrt{-1}\mathfrak{a}_{\varTheta_i}^*)
\]
for some $w\in W$, $i\in\{0,\ldots,r\}$ and $\lambda_{\mathfrak{a}(\varTheta_i)}\in \Dki$.
After \cite{Op99} we call each $L\in\varPi_\bsk$ a tempered residual subspace.
It is easy to see that two distinct tempered residual subspaces are disjoint.
For $i=0,\ldots,r$
let $W(\varTheta_i)$ denote
the stabilizer  of $\mathfrak{a}(\varTheta_i)$ in $W$.
Then $W(\varTheta_i)$ acts on 
$\mathfrak{a}_{\varTheta_i}^*$ 
and if $i<r$ then
it is identified with the Weyl group for the type $BC_{r-i}$ root system
\begin{equation}\label{eq:Ri}
\{\pm \beta_j,\,\pm 2\beta_j,\,\pm(\beta_p\pm \beta_q)\,;\,i<j\leq r,\,i< q<p\leq r\}.
\end{equation}
Put
\begin{align*}
\Cl(\mathfrak{a}_{\varTheta_i,+}^*)
&=\{\lambda\in\mathfrak{a}^*_{\varTheta_i}\,;\,\langle\lambda,\alpha\rangle\ge0 \,\,\text{ for any positive root $\alpha$ in \eqref{eq:Ri}}\} \\
& \simeq\{(\lambda_{i+1},\lambda_{i+2},
\dots,\lambda_r)\in\mathbb{R}^{r-i}\,;\,
0\le\lambda_{i+1}\le \lambda_{i+2}\le \cdots\le\lambda_r\}.
\end{align*}
In particular $\Cl(\mathfrak{a}_{\emptyset,+}^*)=\Cl(\mathfrak{a}_+^*)$
and $\Cl(\mathfrak{a}_{\B,+}^*)=\{0\}$.
The following is immediate.
\begin{lem}\label{lem:specrepr}
The set $\bigcup \varPi_\bsk \subset \acs$ 
%(the whole tempered spectra) 
is stable under $W$.
As a complete set of representatives for the $W$-orbits in $\bigcup \varPi_\bsk$ we can take
\[
\bigsqcup_{i=0}^r \bigl(\Dki+\sqrt{-1}\Cl(\mathfrak{a}_{\varTheta_i,+}^*)\bigr).
\]
\end{lem}

\begin{lem}\label{lem:separations}
For any $L\in\varPi_\bsk$ there exists a polynomial $p_L\in S(\mathfrak a_\C)$
that does not vanish on $L$ and identically vanishes
on each $L'\in \varPi_\bsk\,\setminus\,\{L\}$
with $\dim L'\le \dim L$.
\end{lem}
\begin{proof}
Without loss of generality
we may assume $L=(\eta_1,\ldots,\eta_i)+\sqrt{-1}\mathfrak{a}_{\varTheta_i}^*$
with $(\eta_1,\ldots,\eta_i)\in \Dki$.
Put $
X=\{\xi\in \al-|\be|+1+2\mathbb N\bsk_m+2\mathbb N\,;\,\xi<0\}$.
Then
\[
p_L(\lambda):=\biggl(\prod_{j=1}^i\frac{\prod_{\xi\in X}(\lambda_j^2-\xi^2)}{\lambda_j-\eta_j}\biggr)
\prod_{j=i+1}^r\prod_{\xi\in X}(\lambda_j^2-\xi^2)
\]
has the desired properties.
\end{proof}

For each $\lambda_{\mathfrak{a}(\varTheta_i)}\in\Dki$ ($0\le i\le r$)
let $\mathcal S_{\bsk, \lambda_{\mathfrak{a}(\varTheta_i)}}$
denote the Schwartz space 
on $\lambda_{\mathfrak{a}(\varTheta_i)}+\sqrt{-1}\mathfrak a_{\varTheta_i}^*\simeq\mathbb R^{r-i}$
(the space of rapidly decreasing smooth functions on $\lambda_{\mathfrak{a}(\varTheta_i)}+\sqrt{-1}\mathfrak a_{\varTheta_i}^*$).
By the Euclidean Fourier analysis one sees that
$\bigl\{ \phi|_{\lambda_{\mathfrak{a}(\varTheta_i)}+\sqrt{-1}\mathfrak a_{\varTheta_i}^*} \,;\, \phi\in\PW \bigr\}$
is a dense subspace of $\mathcal S_{\bsk, \lambda_{\mathfrak{a}(\varTheta_i)}}$
with respect to the Schwartz space topology.
Identifying $\mathcal S_{\bsk, \lambda_{\mathfrak{a}(\varTheta_i)}}$
with a space of functions on 
$\bigsqcup_{i=0}^r (\Dki+\sqrt{-1}\mathfrak{a}_{\varTheta_i}^*)$ put
\begin{equation}\label{eq:Schwartz}
\mathcal S_\bsk=\bigoplus_{i=0}^r \bigoplus_{\lambda_{\mathfrak{a}(\varTheta_i)}\in\Dki} \mathcal S_{\bsk, \lambda_{\mathfrak{a}(\varTheta_i)}}^{W(\varTheta_i)}.
\end{equation}
The next lemma will be used in Section \ref{sect:inversion}.
\begin{lem}\label{lem:PWinS}
$\bigl\{ \phi|_{\bigsqcup_{i=0}^r (\Dki+\sqrt{-1}\mathfrak{a}_{\varTheta_i}^*)} \,;\, \phi\in\PW^W \bigr\}$
is a dense subspace of $\mathcal S_\bsk$.
\end{lem}
\begin{proof}
The method below is almost the same as Opdam's proof of \cite[Theorem~5.5]{Op99},
though the claim of the lemma has a subtle difference with what was shown in his proof.
First, clearly $\PW^W$ is identified with a subspace of $\mathcal S_\bsk$ by restriction.
We claim that
for $-1\le i\le r$, $\PW^W$ contains a dense subspace of
\[
\bigoplus_{j=0}^i \bigoplus_{\lambda_{\mathfrak{a}(\varTheta_j)}\in D_\bsk(\varTheta_j)} \mathcal S_{\bsk, \lambda_{\mathfrak{a}(\varTheta_j)}}^{W(\varTheta_j)}.
\]
This is trivial if $i=-1$.
We proceed by induction.
Let $i\ge0$ and assume the claim is true for $i-1$.
Take an arbitrary $\lambda_{\mathfrak{a}(\varTheta_i)}\in\Dki$ and
put $L=\lambda_{\mathfrak{a}(\varTheta_i)}+\sqrt{-1}\mathfrak a_{\varTheta_i}^*$.
Let $p_L$ be the polynomial in Lemma \ref{lem:separations}.
Take an arbitrary $\phi\in C^\infty_0(L)^{W(\varTheta_i)}$.
Then there exists a sequence $\{\psi_n\}\subset \PW$
such that
$\psi_n|_L\to (p_L|_L)^{-1}\phi$ in $\mathcal S_{\bsk, \lambda_{\mathfrak{a}(\varTheta_i)}}$.
Thus
\[
\bigl(|W(\varTheta_i)|^{-1}\sum _{w\in W(\varTheta_i)}p_L(w\lambda) \psi_n(w\lambda)\bigr)\bigr|_L\to \phi
\]
in $\mathcal S_{\bsk, \lambda_{\mathfrak{a}(\varTheta_i)}}$.
Note that
\begin{multline*}
\bigl(|W(\varTheta_i)|^{-1}\sum _{w\in W(\varTheta_i)}p_L(w\lambda) \psi_n(w\lambda)\bigr)\bigr|_L \\
=\bigl(|W^{\lambda(\mathfrak a_{\varTheta_i})}|^{-1}\sum _{w\in W}p_L(w\lambda) \psi_n(w\lambda)\bigr)\bigr|_L
\end{multline*}
and that $|W^{\lambda(\mathfrak a_{\varTheta_i})}|^{-1}\sum _{w\in W}p_L(w\lambda) \psi_n(w\lambda)$ belongs to
\begin{equation}\label{eq:Fkimage}
\PW^W \cap
\biggl(
\mathcal S_{\bsk, \lambda_{\mathfrak{a}(\varTheta_i)}}^{W(\varTheta_i)} \oplus
\bigoplus_{j=0}^{i-1} \bigoplus_{\lambda_{\mathfrak{a}(\varTheta_j)}\in D_\bsk(\varTheta_j)} \mathcal S_{\bsk, \lambda_{\mathfrak{a}(\varTheta_j)}}^{W(\varTheta_j)}
\biggr).
\end{equation}
Here $W^{\lambda_{\mathfrak{a}(\varTheta_i)}}$ is the stabilizer of $\lambda_{\mathfrak{a}(\varTheta_i)}$ in $W$.
Hence by the assumption
there exists a sequence $\{\tau_n\}$ in \eqref{eq:Fkimage}
such that $\tau_n\to \phi$ in $\mathcal S_\bsk$.
This readily proves the claim for $i$.
\end{proof}

\section{
A partial sum of Harish-Chandra series.
}
\label{sect:partial}

For a while we assume $\bsk\in\mathcal K_{\mathrm{reg}}$.
Let $\varXi$ denote the subset of $\B$ 
given by
\begin{equation}\label{eqn:xi}
\varXi=\{\alpha_j\,;\,1\leq j\leq r-1\}.
\end{equation}
Then $\langle\varXi\rangle$ is a root system of type $A_{r-1}$ and $W_\varXi=\mathfrak{S}_r$. 
We define
\begin{equation}\label{eqn:fxi}
{F}_\varXi(\lambda,\bsk)
=\!\sum_{s\,\in \,W_\varXi}\!{\cc}_\varXi(s\lambda,\bsk)\,\varPhi(s\lambda,\bsk). 
\end{equation}
If $r=1$, then $\varXi=\emptyset$. In this case we set $F_\varXi(\lambda,\bsk)=\varPhi(\lambda,\bsk)$, 
$\cc_\varXi(\lambda,\bsk)=1$, and $W_\varXi=\{1_W\}$. 
In this section we rewrite the inverse hypergeometric Fourier transform $\mathcal J_\bsk$
in terms of ${F}_\varXi(\lambda,\bsk)$
and study some properties of ${F}_\varXi(\lambda,\bsk)$.
By the definition, $F_\varXi(s\lambda,\bsk)=F_\varXi(\lambda,\bsk)$ for any $s\in W_\varXi$. 
Note 
\begin{align}\label{eqn:tcfxi}
& \tilde{\cc}_\varXi(\lambda,\bsk)   = 
\prod_{1\leq q<p\leq r}\frac{\varGamma\big(\frac12(\lambda_p-\lambda_q)\big)
}{\varGamma\big(\frac12(\lambda_p-\lambda_q+2\bsk_m)\big)} , \\
& \tilde{\cc}_\varXi(\rho(\bsk),\bsk)=\prod_{j=2}^r \frac{\varGamma(\bsk_m)}{\varGamma(j\bsk_m)}, \label{eqn:clxi}\\
& \tilde{\cc}^\varXi(\lambda,\bsk)   = 
\prod_{1\leq q<p\leq r}\frac{\varGamma\big(\frac12(\lambda_p+\lambda_q)\big)
}{\varGamma\big(\frac12(\lambda_p+\lambda_q+2\bsk_m)\big)} 
 \prod_{j=1}^r \tilde{\cc}_j(\lambda,\bsk). \label{eqn:cuxi}
\end{align}
An argument similar to the proof of \cite[Theorem 4.3.14]{Hec94} shows that the 
singularities of 
${F}_\varXi(\lambda,\bsk)$ are at most simple poles along hyperplanes of the 
form
\begin{equation}\label{eqn:fxihol}
\langle\lambda,\alpha^\vee\rangle=j\quad\text{for some }\,\,\alpha\in\Rp\,\setminus\, \langle\Xi\rangle \,\,
\text{ and } \,\,j=1,2,\dots.
\end{equation}
Such a result for $F_\varTheta$ with a general $\varTheta\subset\mathcal B$ is given by \cite[Theorem~3.5]{P}. 

\begin{rem}
If $\bsk_m=0$, then ${\cc}_\varXi(\lambda,\bsk)={1}/{r!}$, so 
\[
{F}_\varXi(\lambda,\bsk)=
\frac{1}{|W_\varXi|}\sum_{s\in W_\varXi}\!
\varPhi(s\lambda,\bsk).
\]
 Since $\varPhi(\lambda,\bsk)$ is the product of the Harish-Chandra series 
associated with the root systems $\{\pm \beta_i,\pm2\beta_i\}\,\, (1\leq i\leq r)$ in this case
(cf.~\cite[(2-4)]{Sh08}), the above 
statement on regularity is obvious.
\end{rem}

\begin{lem}\label{lem:FXiest}
Let $\bsk\in\mathcal K_{\mathrm{reg}}$ and $x\in\mathfrak a_+$.
For each $\eta$ in 
\begin{equation*}
Z_\varXi=\{\lambda\in\mathfrak a^*\,;\,
\langle\lambda,\alpha^\vee\rangle<1\quad\text{for any }\,\,\alpha\in\Rp\,\setminus\, \langle\Xi\rangle\}
\end{equation*}
there exist a neighborhood $T\subset Z_\varXi$ of $\eta$
and constants $C>0$, $n\in\mathbb N$ such that
\begin{equation*}
|F_\varXi(\lambda,\bsk;x)|\le C(1+||\lambda||)^n\quad\text{for }\lambda\in T+\sqrt{-1}\mathfrak a^*.
\end{equation*}
\end{lem}
\begin{proof}
As a function in $\lambda$, $F_\varXi(\lambda,\bsk;x)$ is holomorphic on $Z_\varXi+\sqrt{-1}\mathfrak a^*$.
Take any compact neighborhood $T\subset Z_\varXi$ of $\eta$.
We assert that
there exists a sequence of pairs $\{(v_j, a_j)\}_{j=1}^k\subset (\mathfrak a\,\setminus\,\{0\})\times\mathbb R$
such that
\[
\psi(\lambda):=F_\varXi(\lambda,\bsk;x)\prod_{j=1}^k(\lambda(v_j)-a_j)
\]
satisfies
\begin{equation}\label{eq:FXiest}
\exists C>0\ \exists n\in\mathbb N\ \forall\lambda\in T+\sqrt{-1}\mathfrak a^*\quad
|\psi(\lambda)|\le C(1+||\lambda||)^{n}.
\end{equation}
Indeed, we can take a sufficiently large $N\in\mathbb N$ so that
\[
\prod_{1\leq q<p\leq r}\frac{\varGamma\big(\frac12(\lambda_p-\lambda_q)+N\big)
}{\varGamma\big(\frac12(\lambda_p-\lambda_q+2\bsk_m)\big)}
\]
is regular on $W_\varXi T+\sqrt{-1}\mathfrak a^*$.
Put
\[
P(\lambda)=p(\lambda)\,
\tilde\cc_\varXi(\lambda,\bsk)^{-1}\!\prod_{1\leq q<p\leq r}\frac{\varGamma\big(\frac12(\lambda_p-\lambda_q)+N\big)
}{\varGamma\big(\frac12(\lambda_p-\lambda_q+2\bsk_m)\big)}
\]
where $p(\lambda)$ is the polynomial in Lemma \ref{lem:esthcs} for the compact set $W_\varXi T\subset\mathfrak a^*$.
Choose $\{(v_j, a_j)\}_{j=1}^k$ so that
\[
\prod_{j=1}^k(\lambda(v_j)-a_j)=\prod_{s\in W_\varXi} P(s\lambda).
\]
Then the assertion follows from Lemma \ref{lem:esthcs}, Lemma \ref{lem:estcf} and \eqref{eqn:fxi}.

If $k=0$ then we are done. So assume $k>0$.
Let us show that
even if $\psi(\lambda)$ is replaced with
\[
\phi(\lambda):=\frac{\psi(\lambda)}{\lambda(v_1)-a_1}=F_\varXi(\lambda,\bsk;x)\prod_{j=2}^{k}(\lambda(v_j)-a_j),
\]
\eqref{eq:FXiest} still holds for a smaller $T$.
Since this is obvious when $\eta(v_1)\ne a_1$,
we assume $\eta(v_1)= a_1$.
We identify $v_1$ with a vector in $\mathfrak a^*$
and take a small $\varepsilon>0$ so that
\[
T':=\eta+\{\zeta v_1\,;\,-\varepsilon\le\zeta\le\varepsilon\}+\{\lambda\in\mathfrak a^*\,;\,\langle\lambda,v_1\rangle=0, ||\lambda||\le\varepsilon\}\subset T.
\]
Now let $\lambda\in T'+\sqrt{-1}\mathfrak a^*$.
Suppose first $|\lambda(v_1)-a_1|\ge\varepsilon||v_1||^2$.
Then we have $|\phi(\lambda)|\le\varepsilon^{-1}||v_1||^{-2}|\psi(\lambda)|\le \varepsilon^{-1}||v_1||^{-2}C(1+||\lambda||)^n$.
Next, suppose $|\lambda(v_1)-a_1|\le\varepsilon||v_1||^2$
and put 
\[
\lambda_\zeta=\lambda+\biggl(\zeta-\frac{\lambda(v_1)-a_1}{||v_1||^2}\biggr)v_1
\]
for $\zeta\in\C$ with $|\zeta|\le\varepsilon$.
Then $||\lambda_\zeta||\le||\lambda||+2\varepsilon||v_1||$,
$\lambda_\zeta(v_1)-a_1=\zeta||v_1||^2$
and $\lambda_\zeta\in T'+\sqrt{-1}\mathfrak a^*$.
Hence by the maximum modulus principle
\[
\begin{aligned}
|\phi(\lambda)|&\le \max_{|\zeta|\le\varepsilon}|\phi(\lambda_\zeta)|=\max_{|\zeta|=\varepsilon}|\phi(\lambda_\zeta)|
=\varepsilon^{-1}||v_1||^{-2}\max_{|\zeta|=\varepsilon}|\psi(\lambda_\zeta)|\\
&\le \varepsilon^{-1}||v_1||^{-2} C(1+2\varepsilon||v_1||+||\lambda||)^n.
\end{aligned}
\]
Thus our claim is shown. Repeating the same argument
we can remove all the factors $(\lambda(v_j)-a_j)$ from $\psi(\lambda)$.
\end{proof}

\begin{prop}[Inversion formula, second form]\label{prop:second}
Let $\bsk\in\mathcal{K}'_1$.
If $r>1$ then we further assume $\bsk_m>0$.
Let $\phi\in\PW^W$ and $x\in\mathfrak a_+$.
Then
\begin{multline*}
\mathcal{J}_\bsk\,\phi(x)-\mathcal{J}_{\bsk,\emptyset}\,\phi(x)\\
=\frac{-1}{(2\pi\sqrt{-1})^{r-1}(r-1)!}
\sum_{\xi\in \Dko}
\int_{\check{\xi}+\sqrt{-1}\mathfrak{a}_{\varTheta_1}^*}\!
\left.(\phi(\lambda){F}_\varXi(\lambda,\bsk;x))\right|_{\lambda_1=\xi}\\
\times \Res_{\lambda_1=\xi}({\cc}_\varXi(\lambda,\bsk)^{-1}\cc(-\lambda,\bsk)^{-1})
\,d\lambda_{\mathfrak{a}_{\varTheta_1}}.
\end{multline*}
Here $d\lambda_{\mathfrak{a}_{\varTheta_1}}=d\lambda_2d\lambda_3\cdots d\lambda_r$ and
$\check{\xi}:=(\xi,\xi,\dots,\xi)\in\mathbb R^{r-1}\simeq\mathfrak{a}^*_{\varTheta_1}$
for $\xi\in \Dko=\{\xi\in\mathbb{R}\,;\,\xi<0,\,\xi-\al+|\be|-1\in 2\mathbb{N}\}$.
 (If $r=1$ then $\check{\xi}=0$ for each $\xi\in \Dko$ and
 $d\lambda_{\mathfrak{a}_{\varTheta_1}}$ is the counting measure on $\check{\xi}+\sqrt{-1}\mathfrak{a}_{\varTheta_1}^*=\{0\}$.) 
\end{prop}

\begin{proof}
The singularities of the function 
$\lambda\mapsto\cc(-\lambda,\bsk)^{-1}$ in the region $\{\lambda\in\acs\,;\,
\text{Re}\,\lambda\in-\Cl(\mathfrak{a}_+^*)\}$ are precisely 
the simple poles along
$\lambda_j=\xi$ for $1\leq j\leq r$ and $\xi\in \Dko$
(cf.~\eqref{eqn:tcfbc1}). 
Let $\{\xi^{(1)}<\xi^{(2)}<\cdots<\xi^{(k)}=0\}=\Dko\cup\{0\}$.
With a sufficiently small $\varepsilon>0$ 
put
\begin{align*}
\eta^{(1,0)}&=(\xi^{(1)}-\varepsilon,\,\xi^{(1)}-\varepsilon,\,\xi^{(1)}-\varepsilon,\ldots,\,\xi^{(1)}-\varepsilon),\\
\eta^{(1,1)}&=(\xi^{(2)}-\varepsilon,\,\xi^{(1)}-\varepsilon,\,\xi^{(1)}-\varepsilon,\ldots,\,\xi^{(1)}-\varepsilon),\\
\eta^{(1,2)}&=(\xi^{(2)}-\varepsilon,\,\xi^{(2)}-\varepsilon,\,\xi^{(1)}-\varepsilon,\ldots,\,\xi^{(1)}-\varepsilon),\\
&\  \,\vdots\\
\eta^{(1,r)}=\eta^{(2,0)}&=(\xi^{(2)}-\varepsilon,\,\xi^{(2)}-\varepsilon,\,\xi^{(2)}-\varepsilon,\ldots,\,\xi^{(2)}-\varepsilon),\\
\eta^{(2,1)}&=(\xi^{(3)}-\varepsilon,\,\xi^{(2)}-\varepsilon,\,\xi^{(2)}-\varepsilon,\ldots,\,\xi^{(2)}-\varepsilon),\\
&\ \,\vdots\\
\eta^{(k-1,r)}=\eta^{(k,0)}&=(\xi^{(k)}-\varepsilon,\,\xi^{(k)}-\varepsilon,\,\xi^{(k)}-\varepsilon,\ldots,\,\xi^{(k)}-\varepsilon)\\
&=(-\varepsilon,\,-\varepsilon,\,-\varepsilon,\ldots,\,-\varepsilon).
\end{align*}
All these points belong to $-\Cl(\mathfrak a_+^*)$.
Define
\[
I^{(\l,j)}=\int_{\eta^{(\l,j)}+\sqrt{-1}\mathfrak{a}^*}\phi(\lambda)\,\varPhi(\lambda,\bsk;x)\,\cc(-\lambda,\bsk)^{-1}
d\mu(\lambda).
\]
Then $I^{(1,0)}=\mathcal{J}_\bsk\,\phi(x)$ and $I^{(k,0)}=\mathcal{J}_{\bsk,\emptyset}\,\phi(x)$.

Let $1\le\l<k$ and $1\le j \le r$.
By Cauchy's residue theorem, 
$I^{(\l,j)}-I^{(\l,j-1)}$ is
\begin{align}\label{eqn:res1}
\frac{1}{(2\pi\sqrt{-1})^{r-1}}
\int_{{\xi}^{(\l,j)}+\sqrt{-1}\mathbb{R}^{r-1}} & (\phi(\lambda)\,\varPhi(\lambda,\bsk;x))|_{\lambda_j=\xi^{(\l)}} \\
 \times \,\Res_{\lambda_j=\xi^{(\l)}}&(\cc(-\lambda,\bsk)^{-1})\,
d\lambda_{\{1,2,\dots,r\}\setminus\{j\}}, \notag
\end{align}
where
\[
{\xi}^{(\l,j)}
=\bigl((\xi^{(\l+1)}-\varepsilon,\ldots,\xi^{(\l+1)}-\varepsilon),\,
(\xi^{(\l)}-\varepsilon,\ldots,\xi^{(\l)}-\varepsilon)\bigr)
\in \mathbb R^{j-1}\times \mathbb R^{r-j}
\]
and $d\lambda_{\{1,2,\dots,r\}\setminus\{j\}}=\prod_{i\not=j}d\lambda_i$. 
The integrand as a function of $(\lambda_i)_{i\,\in\, \{1,2,\dots,r\}\setminus\{j\}}$ 
is regular near
$\check{\xi}^{(\l)}+\sqrt{-1}\mathbb{R}^{r-1}$
because the pole of  \,$\tilde{\cc}_i(-\lambda,\bsk)^{-1}\bigr|_{\lambda_j=\xi^{(\l)}}$ along $\lambda_i=\xi^{(\l)}\,\,(i\not=j)$ is canceled by a zero locus of 
$\tilde{\cc}_\alpha(-\lambda,\bsk)^{-1}\bigr|_{\lambda_j=\xi^{(\l)}}$
for either $\alpha=\beta_i-\beta_j$ or $\beta_j-\beta_i$ in $ \Rp$
(recall the assumption that $\bsk_m>0$ if $r>1$). 
Hence we can shift the domain of integration in \eqref{eqn:res1}
from ${\xi}^{(\l,j)}+\sqrt{-1}\mathbb{R}^{r-1}$ to $\check{\xi}^{(\l)}+\sqrt{-1}\mathbb{R}^{r-1}$.

Now the function $\lambda\mapsto {\cc}_\varXi(\lambda,\bsk)^{-1}$ is regular on
\begin{equation}\label{eqn:clnegch}
\{\lambda\in\mathbb{C}^r\,;\,\mathrm{Re}\,\lambda_1\leq \mathrm{Re}\,\lambda_2\leq \cdots 
\leq \mathrm{Re}\,\lambda_r\leq 0\}.
\end{equation}
Also, \eqref{eqn:tcf} and \eqref{eqn:tcfxi} yield
\begin{equation}\label{eq:WXiinvc}
{\cc}_\varXi(s\lambda,\bsk)^{-1}\cc(-s\lambda,\bsk)^{-1}
={\cc}_\varXi(\lambda,\bsk)^{-1}\cc(-\lambda,\bsk)^{-1}\quad \text{for }
s\in W_\varXi.
\end{equation}
Therefore, by \eqref{eqn:fxi} and changes of variables,
$I^{(\l+1,0)}-I^{(\l,0)}=\sum_{j=1}^{r}(I^{(\l,j)}-I^{(\l,j-1)})$ reduces to
\begin{align}\label{eqn:int1}
\frac{1}{(2\pi\sqrt{-1})^{r-1}(r-1)!} & 
\int_{\check{\xi}^{(\l)}+\sqrt{-1}\mathfrak{a}_{\varTheta_1}^*}\!
\left.(\phi(\lambda){F}_\varXi(\lambda,\bsk;x))\right|_{\lambda_1=\xi^{(\l)}} \\
& \quad \times \Res_{\lambda_1=\xi^{(\l)}}({\cc}_\varXi(\lambda,\bsk)^{-1}\cc(-\lambda,\bsk)^{-1})
\,d\lambda_{\mathfrak{a}_{\varTheta_1}}.  \notag
\end{align}
This proves the proposition.
\end{proof}

Suppose $\bsk\in\mathcal{K}'_1$, $0\le i\le r$ and $\xi\in\Dki$.
For a meromorphic function $\phi(\lambda)$ on $\acs$,
we define its successive restriction
\begin{equation}\label{eq:suc_res}
\sigma_{\xi}\bigl(\phi(\lambda)\bigr)=\phi(\lambda)\bigr|_{\lambda_1=\xi_1}\bigr|_{\lambda_2=\xi_2}\cdots
\bigr|_{\lambda_i=\xi_i}
\end{equation}
when the right hand side is justified.
(If $i=0$ and $\xi=0$ then $\sigma_{\xi}\bigl(\phi(\lambda)\bigr)=\phi(\lambda)$.)

\begin{lem}\label{lem:crestr}
Both
$\sigma_{\xi}\bigl(\cc^\varXi(\lambda,\bsk)^{-1}\bigr)$
and $\sigma_{\xi}\bigl(\cc^\varXi(\lambda,\bsk)\bigr)$ 
are well-defined meromorphic functions on $\mathfrak a_{\varTheta_i,\C}^*$.
More precisely, the latter can be written as
\begin{multline*}
\sigma_{\xi}\bigl(\cc^\varXi(\lambda,\bsk)\bigr)
=a(\xi,\bsk;\lambda_{\mathfrak a_{\varTheta_i}})
\prod_{1\leq q\le i<p\leq r}
\frac{\varGamma\big(\frac12(\lambda_p+\xi_q)\big)}
{\varGamma\big(\frac12(\lambda_p+\xi_q+2\bsk_m)\big)}\\
\times\prod_{i<q<p\leq r}
\frac{\varGamma\big(\frac12(\lambda_p+\lambda_q)\big)}
{\varGamma\big(\frac12(\lambda_p+\lambda_q+2\bsk_m)\big)}
\prod_{j=i+1}^r\frac{\varGamma\big(\lambda_j\big)}{\varGamma\big(\tfrac12(\lambda_{j}+\al-|\be|+1)\big)}
\end{multline*}
with a meromorphic function $a(\xi,\bsk;\lambda_{\mathfrak a_{\varTheta_i}})$ that is regular and non-vanishing on a neighborhood of
the region $\{\lambda_{\mathfrak a_{\varTheta_i}}=(\lambda_{i+1},\ldots,\lambda_r)\in\C^{r-i}\,;\,
\al-|\be|+1\le \operatorname{Re}\lambda_j\le0\ (i< j\le r)\}$.
\end{lem}
\begin{proof}
If $i=0$ then the lemma is obvious by \eqref{eqn:tcfbc1} and \eqref{eqn:cuxi}.
Assume the lemma is true for $i$
and let $(\xi,\xi_{i+1})\in D_\bsk(\varTheta_{i+1})$.
If $i=0$ then the restriction of
\[
\frac{\varGamma\big(\lambda_{1}\big)}{\varGamma\big(\tfrac12(\lambda_{1}+\al-|\be|+1)\big)}
\]
to $\lambda_1=\xi_1$
is a non-zero number $C$
because for $\xi_{1}\in D_\bsk(\varTheta_{1})$
\[
\xi_{1}\in -\mathbb N  \Leftrightarrow \xi_{1}+\al-|\be|+1\in-2\mathbb N.
\]
If $i>1$ then
\begin{multline*}
\frac{\varGamma\big(\lambda_{i+1}\big)}{\varGamma\big(\tfrac12(\lambda_{i+1}+\al-|\be|+1)\big)}
\prod_{1\leq q\le i}
\frac{\varGamma\big(\frac12(\lambda_{i+1}+\xi_q)\big)}
{\varGamma\big(\frac12(\lambda_{i+1}+\xi_q+2\bsk_m)\big)}\\
=
\frac{\varGamma\big(\frac12(\lambda_{i+1}+\xi_{1})\big)}{\varGamma\big(\tfrac12(\lambda_{i+1}+\al-|\be|+1)\big)}
\times
\frac{\varGamma\big(\lambda_{i+1}\big)}
{\varGamma\big(\frac12(\lambda_{i+1}+\xi_i+2\bsk_m)\big)}\\
\times\prod_{1\leq q< i}
\frac{\varGamma\big(\frac12(\lambda_{i+1}+\xi_{q+1})\big)}
{\varGamma\big(\frac12(\lambda_{i+1}+\xi_q+2\bsk_m)\big)}
\end{multline*}
and the restriction of this function to $\lambda_{i+1}=\xi_{i+1}$
is a non-zero number $C$
because for $\xi_{i+1}\in\C$ with $(\xi,\xi_{i+1})\in D_\bsk(\varTheta_{i+1})$ we have
\begin{align*}
\xi_{i+1}+\xi_1\in -2\mathbb N & \Leftrightarrow \xi_{i+1}+\al-|\be|+1\in-2\mathbb N,\\
\xi_{i+1}\in -\mathbb N & \Leftrightarrow \xi_{i+1}+\xi_i+2\bsk_m \in -2\mathbb N,\\
\xi_{i+1}+\xi_{q+1}\in -2\mathbb N &
\Leftrightarrow \xi_{i+1}+\xi_q+2\bsk_m \in -2\mathbb N\quad(1\le q<i). 
\end{align*}
Thus, in either case, the lemma is also true for $i+1$ with 
\[
a((\xi,\xi_{i+1}),\bsk;\lambda_{\mathfrak a_{\varTheta_{i+1}}})
=C\,a(\xi,\bsk;\lambda_{\mathfrak a_{\varTheta_i}})\bigr|_{\lambda_{i+1}=\xi_{i+1}}.
\qedhere
\]
\end{proof}

We define the subgroup $W^\varXi:=\mathbb{Z}_2^r$ of $W$.
This is a complete set of representatives 
for both $W/W_\varXi$ and $W_\varXi\backslash W$. 
By \eqref{eqn:hcs2} and \eqref{eq:updaownTheta}, 
\begin{equation}\label{eqn:phc}
F(\lambda,\bsk)=\!
\sum_{w\,\in\, 
W^\varXi
}
\!{\cc}^\varXi(w\lambda,\bsk){F}_\varXi(w\lambda,\bsk).
\end{equation}
For $0\leq i\leq r$ define $W^{\varXi}_i=W^\varXi\cap W({\varTheta_i})$.
Then $W^{\varXi}_i\simeq \mathbb{Z}_2^{r-i}$, whose elements 
act as changes of signs for $\beta_{i+1},\dots,\beta_r$.  

\begin{prop}\label{prop:xitempered}
Suppose 
$\bsk\in\mathcal{K}'_1$, $0\leq i\leq r$, 
$x\in\mathfrak a_+$ and $\xi\in \Dki$.
Then as a meromorphic function on
$\mathfrak a_{\varTheta_i,\C}^*$,
\begin{equation}\label{eqn:expxip}
F(\lambda,\bsk;x)\bigr|_{\lambda_{\mathfrak a({\varTheta_i})}=\xi}=\!\sum_{w\,\in\, W^{\varXi}_i}
\sigma_{\xi}\bigl(\cc^\varXi(w\lambda,\bsk)\bigr)\,
F_\varXi(w\lambda,\bsk;x)\bigr|_{\lambda_{\mathfrak a({\varTheta_i})}=\xi}\,.
\end{equation}
Moreover,
the coefficient of each term on the right hand side of \eqref{eqn:expxip}
is regular and non-zero at $\lambda_{\mathfrak{a}_{\varTheta_i}}\in\sqrt{-1}\mathfrak{a}_{\varTheta_i}^*$
such that
\begin{equation}\label{eqn:supgeneric}
\lambda_j\ne 0\ (i<\forall j\le r)\text{ and }\lambda_p\pm\lambda_q\ne 0\ (i<\forall q<\forall p\le r).
\end{equation} 
\end{prop}

\begin{proof}
Since the restriction operator $\sigma_\xi$ and the $W_i^\varXi$-action commute,
the second assertion of the proposition follows from Lemma \ref{lem:crestr}.

We prove \eqref{eqn:expxip} in a similar way as Theorem~\ref{thm:temperedtheta1}.
If $\Delta\xi\in \mathfrak{a}(\varTheta_i)_{\mathbb{C}}^*$
varies with $\Delta\bsk\in \mathcal K_{\mathbb C}$ as
$\Delta\lambda_{\mathfrak{a}(\varTheta_i)}$ does in Remark \ref{rem:regularization},
then it is easy to see from the proof of Lemma \ref{lem:crestr} 
that $a(\xi+\Delta\xi,\bsk+\Delta\bsk;\lambda_{\mathfrak a_{\varTheta_i}})$
is holomorphic in $\Delta\bsk$ around $0$.
Thus, for any fixed $\lambda_{\mathfrak{a}_{\varTheta_i}}\in\sqrt{-1}\mathfrak{a}_{\varTheta_i}^*$ satisfying \eqref{eqn:supgeneric},
$\sigma_{\xi}\bigl(\cc^\varXi(w\lambda,\bsk)\bigr)$ $(w\in W_i^\varXi)$
analytically depends on $\bsk$.
Hence we have only to prove \eqref{eqn:expxip}
for generic $\bsk$ and $\lambda_{\mathfrak{a}_{\varTheta_i}}$ such that
\eqref{eqn:condlambda} holds for $\lambda=\xi+\lambda_{\mathfrak{a}_{\varTheta_i}}$.
For such $\bsk$ and $\lambda_{\mathfrak{a}_{\varTheta_i}}$,
\[
\sigma_{\xi}\bigl(\cc^\varXi(w\lambda,\bsk)\bigr)
=\cc^\varXi(w(\xi+\lambda_{\mathfrak{a}_{\varTheta_i}}),\bsk)\quad
(w\in W_i^\varXi)
\]
and \eqref{eqn:phc} holds for $\lambda=\xi+\lambda_{\mathfrak{a}_{\varTheta_i}}$.
Thus \eqref{eqn:expxip} follows if we can show
\begin{equation*}
\cc^\varXi(w(\xi+\lambda_{\mathfrak{a}_{\varTheta_i}}),\bsk)=0\quad
\text{for any }w\in W^\varXi\,\setminus\, W^\Xi_i.
\end{equation*}
Let $w\in W^\varXi\,\setminus\, W^\Xi_i$.
Then there exists $j\in\{1,\ldots,i\}$
such that $w\beta_j=-\beta_j$.
Let $j$ be the smallest one.
If $j=1$ then $\cc^\varXi(w(\xi+\lambda_{\mathfrak{a}_{\varTheta_i}}),\bsk)=0$
by \eqref{eqn:tcfbc1} and \eqref{eqn:D}.
If $j>1$ then 
$\cc^\varXi(w(\xi+\lambda_{\mathfrak{a}_{\varTheta_i}}),\bsk)$ contains the factor 
\[
\frac{\varGamma\big(\frac12(-\xi_j+\xi_{j-1})\big)
}{\varGamma\big(\frac12(-\xi_j+\xi_{j-1}+2\bsk_m)\big)}, 
\]
which vanishes by \eqref{eqn:D}.
\end{proof}

\begin{rem}
Corollary~\ref{cor:tempered2} follows also from Proposition~\ref{prop:xitempered}.
\end{rem}

\section{
Inversion and Plancherel formula
}
\label{sect:inversion}

In this section, we assume $\bsk\in\mathcal{K}'_1$,
namely, $\bsk_s$, $\bsk_m$ and $\bsk_\ell$ are real numbers
satisfying \eqref{eqn:mcond2}.
Recall $\al$ and $\be$ are given by \eqref{eqn:paraj}. 
The sets $\Dki\subset \mathfrak{a}(\varTheta_i)^*\,\,(0\leq i\leq r)$ are defined by \eqref{eqn:D}. 
For $i=0,\ldots,r$ and $\lambda_{\mathfrak{a}(\varTheta_i)}=(\lambda_1,\ldots,\lambda_i)\in \Dki$
we define a positive number $d_{\varTheta_i} (\lambda_{\mathfrak{a}(\varTheta_i)},\bsk)$ as follows.
If $\bsk_m>0$ then we put
\begin{align}\label{eqn:fdefd}
&d_{\varTheta_i} (\lambda_{\mathfrak{a}(\varTheta_i)},\bsk)
 =
\tilde{\cc}_{\varTheta_i}(\rho(\bsk),\bsk)^2\\
&  \times \!
\prod_{j=1}^i \frac{-2^{2\al-2\be-1} \lambda_j}{\pi}\!
\frac{\varGamma\big(\frac12(\lambda_j+\al+|\be|+1)\big)\varGamma\big(\frac12(-\lambda_j+\al+|\be|+1)\big)}
{\varGamma\big(\frac12(\lambda_j-\al+|\be|+1)\big)\varGamma\big(\frac12(-\lambda_j-\al+|\be|+1)\big)}
\notag \\
& 
\times \!
\prod_{1\leq q<p\leq i}\!
\frac{(\lambda_q^2-\lambda_p^2) 
\varGamma\big(\frac12(\lambda_p-\lambda_q+2\bsk_m)\big)\varGamma\big(\frac12(-\lambda_q-\lambda_p+2\bsk_m)\big)}
{4\,\varGamma\big(\frac12(\lambda_p-\lambda_q-2\bsk_m+2)\big)\varGamma\big(\frac12(-\lambda_q-\lambda_p-2\bsk_m+2)\big)}.
\notag
\end{align}
If $\bsk_m=0$ or $r=1$ we put
\begin{align}\label{eqn:fdefd2}
&d_{\varTheta_i} (\lambda_{\mathfrak{a}(\varTheta_i)},\bsk)
 =
\tilde{\cc}_{\varTheta_i}(\rho(\bsk),\bsk)^2
\bigl|W_{\varTheta_i}^{\lambda_{\mathfrak{a}(\varTheta_i)}}\bigr|^{-1}\\
&  \times \!
\prod_{j=1}^i \frac{-2^{2\al-2\be-1} \lambda_j}{\pi}\!
\frac{\varGamma\big(\frac12(\lambda_j+\al+|\be|+1)\big)\varGamma\big(\frac12(-\lambda_j+\al+|\be|+1)\big)}
{\varGamma\big(\frac12(\lambda_j-\al+|\be|+1)\big)\varGamma\big(\frac12(-\lambda_j-\al+|\be|+1)\big)}
\notag
\end{align}
where $W_{\varTheta_i}^{\lambda_{\mathfrak{a}(\varTheta_i)}}$ is the stabilizer of $\lambda_{\mathfrak{a}(\varTheta_i)}$ in $W_{\varTheta_i}$.
Observe that $d_{\varTheta_i} (\lambda,\bsk)$ is actually positive.
In particular $d_\emptyset(0,\bsk)=1$.
\begin{rem}
If $\lambda_{\mathfrak{a}(\varTheta_i)}$ varies with $\bsk$ as in Remark \ref{rem:regularization}
then \eqref{eqn:fdefd2} is the limit of \eqref{eqn:fdefd} as $\bsk_m\to 0^+$.
\end{rem}

\begin{lem}\label{lem:cTheta}
Suppose $1\le i<r$ and $\lambda_{\mathfrak{a}(\varTheta_i)}=(\lambda_1,\ldots,\lambda_i)\in \Dki$.
If we write $\lambda=\lambda_{\mathfrak{a}(\varTheta_i)}+\lambda_{\mathfrak{a}_{\varTheta_i}}$
for $\lambda_{\mathfrak{a}_{\varTheta_i}}=(\lambda_{i+1},\ldots,\lambda_r)\in \C^{r-i}\simeq\mathfrak a_{\varTheta_i,\C}^*$ 
then $\cc^{\varTheta_i}(-\lambda,\bsk)^{-1}$
is a well-defined meromorphic function in $\lambda_{\mathfrak{a}_{\varTheta_i}}$.
The singularities of this function in the region 
\begin{equation}\label{eqn:Zreg}
\left\{
\lambda_{\mathfrak{a}_{\varTheta_i}}
\in\mathbb{C}^{r-i}\,;
\,\begin{aligned}
&\lambda_i\le \operatorname{Re}\lambda_{i+1}\leq \cdots 
\leq \operatorname{Re}\lambda_r\leq 0,\text{ and}\\
&\operatorname{Re}\lambda_r-\operatorname{Re}\lambda_{i+1}<2\bsk_m \text{ when $\bsk_m>0$}
\end{aligned}
\right\}
\end{equation}
are precisely 
the simple poles along
$\lambda_j=\xi$ for $j=i+1,\ldots, r$ and $\xi\in \mathbb R$
such that $(\lambda_1,\ldots,\lambda_i,\xi)\in D_\bsk(\varTheta_{i+1})$.
In particular $\cc^{\varTheta_i}(-\lambda,\bsk)^{-1}$ is regular on $\sqrt{-1}\mathfrak a_{\varTheta_i}^*$.
\end{lem}
\begin{proof}
The first assertion is clear from the definition of $\cc^{\varTheta_i}(-\lambda,\bsk)$.
By \eqref{eqn:tcfbc1}
the function
\begin{multline}\label{eq:crutialprod}
\prod_{j=i+1}^r\frac{\varGamma\big(\frac12(-\lambda_j+\al-|\be|+1)\big)
\varGamma\big(\frac12(-\lambda_j+\al+|\be|+1)\big)}{\varGamma(-\lambda_j)}\\
\times \prod_{\substack{1\leq q<p\leq r \\ p>i}}\frac{\varGamma\big(\frac12(-\lambda_q-\lambda_p+2\bsk_m)\big)
\varGamma\big(\frac12(\lambda_q-\lambda_p+2\bsk_m)\big)}{\varGamma\big(\frac12(-\lambda_q-\lambda_p)\big)
\varGamma\big(\frac12(\lambda_q-\lambda_p)\big)}
\end{multline}
has the same singularities as ${\cc}^{\varTheta_i}(-\lambda,\bsk)^{-1}$.
We assert for $j=i+1,\ldots,r$ that
the singularities of the partial product
\begin{equation}\label{eqn:singpart1}
\frac{{\varGamma\big(\tfrac12(-\lambda_{j}+\al-|\be|+1)\big)}}{\varGamma(-\lambda_j)}
\prod_{q=1}^i\frac{
\varGamma\big(\frac12(\lambda_q-\lambda_{j}+2\bsk_m)\big)}{
\varGamma\big(\frac12(\lambda_q-\lambda_{j})\big)} 
\end{equation}
on the region \eqref{eqn:Zreg}
are precisely the simple poles along
$\lambda_j=\xi$ for $\xi\in \mathbb R$
such that $(\lambda_1,\ldots,\lambda_i,\xi)\in D_\bsk(\varTheta_{i+1})$.
In fact, we can rewrite \eqref{eqn:singpart1} as the product of
\begin{equation}\label{eqn:singpart3}
\frac{\varGamma\big(\tfrac12(-\lambda_{j}+\al-|\be|+1)\big)}{\varGamma\big(\tfrac12(-\lambda_{j}+\lambda_1)\big)}
\prod_{q=1}^{i-1}\frac{
\varGamma\big(\tfrac12(-\lambda_{j}+\lambda_q+2\bsk_m)\big)}{
\varGamma\big(\tfrac12(-\lambda_{j}+\lambda_{q+1})\big)}
\end{equation}
and
\begin{equation}\label{eqn:singpart2}
\frac{\varGamma\big(\tfrac12(-\lambda_{j}+\lambda_i+2\bsk_m)\big)}{\varGamma(-\lambda_j)}.
\end{equation}
On the one hand, \eqref{eqn:singpart3} is regular and non-vanishing on \eqref{eqn:Zreg}
since for $\lambda_j\in\C$ with $\lambda_i\le\operatorname{Re}\lambda_j$ we have
\begin{align*}
-\lambda_{j}+\al-|\be|+1 \in -2\mathbb N &\Leftrightarrow
-\lambda_{j}+\lambda_1 \in -2\mathbb N,\\
-\lambda_{j}+\lambda_q+2\bsk_m  \in -2\mathbb N
&\Leftrightarrow -\lambda_{j}+\lambda_{q+1} \in -2\mathbb N\quad(1\le q <i).
\end{align*}
On the other hand, the poles of \eqref{eqn:singpart2} are just as stated above.

It is easy to check that
the other factors in \eqref{eq:crutialprod}
produce neither any pole nor any zero locus that cancels a pole of \eqref{eqn:singpart1}.
\end{proof}

Let 
\[
d\lambda_{\mathfrak{a}_{\varTheta_i}}=d\lambda_{i+1}\cdots d\lambda_r
\]
 denote 
the Euclidean measure on $\mathfrak{a}_{\varTheta_i}^*$ and  $\mu_{\varTheta_i}$  the 
measure on $\sqrt{-1}\mathfrak{a}_{\varTheta_i}^*$ 
given by 
\begin{equation}
d\mu_{\varTheta_i}(\lambda_{\mathfrak{a}_{\varTheta_i}})=(2\pi)^{-r+i}\,d(\mathrm{Im}\, \lambda_{\mathfrak{a}_{\varTheta_i}^*}\!)
=(2\pi\sqrt{-1})^{-r+i}\,d \lambda_{\mathfrak{a}_{\varTheta_i}},
\end{equation}
which coincides with \eqref{eqn:measure1} if $i=0$. 

For $0\leq i\leq r$, let $\nu_{\bsk,\varTheta_i}$ denote the measure on 
$\Dki+\sqrt{-1}\mathfrak{a}_{\varTheta_i}^*$ defined by
\begin{multline}\label{eqn:meas00i}
\int_{\Dki+\sqrt{-1}\mathfrak{a}_{\varTheta_i}^*} \!\!
 \psi(\lambda)  \,d \nu_{\bsk,\varTheta_i}(\lambda) \\ 
=
\sum_{\lambda_{\mathfrak{a}(\varTheta_i)}
\,\in \,\Dki}\!\!d_{\varTheta_i} (\lambda_{\mathfrak{a}(\varTheta_i)},\bsk) 
 \int_{\sqrt{-1}\mathfrak{a}_{\varTheta_i}^*}\!\!
\psi(\lambda)\,|\cc^{\varTheta_i}(\lambda,\bsk)|^{-2}
\,
d\mu_{\varTheta_i}(\lambda_{\mathfrak{a}_{\varTheta_i}}).
\end{multline}
In particular, 
\[
\int_{\sqrt{-1}\mathfrak{a}^*}\psi(\lambda)\,d\nu_{\bsk,\emptyset}(\lambda)=
\int_{\sqrt{-1}\mathfrak{a}^*}\!\psi(\lambda)\,\frac{d\mu(\lambda)}{|\cc(\lambda,\bsk)|^2}
\]
and
\[
\int_{\Dkr}
 \psi(\lambda)\,d\nu_{\bsk,\B}(\lambda)=\!
\sum_{\lambda\,\in\, \Dkr}\!d_{\B}(\lambda,\bsk)\,\psi(\lambda).
\]

Recall $\mathcal S_\bsk$ is defined by \eqref{eq:Schwartz}.
For $\phi\in\mathcal S_\bsk$ let $\mathcal{J}_{\bsk,\varTheta_i}\phi(x)\,\,(0\leq i\leq r)$ denote the functions defined by
\begin{equation}\label{eqn:jfb0}
\mathcal{J}_{\bsk,\varTheta_i} \phi(x)=
\frac{1}{|W(\varTheta_i)|}\int_{\Dki+\sqrt{-1}\mathfrak{a}_{\varTheta_i}^*}
\phi(\lambda)F(\lambda,\bsk;x)\,d\nu_{\bsk,\varTheta_i}(\lambda).
\end{equation}
By \eqref{eq:PWcondition}, Theorem \ref{lem:est}, Lemma \ref{lem:cTheta} and Lemma \ref{lem:estcf}, 
the integral on the right hand side of 
\eqref{eqn:jfb0}  converges and defines an element of $C^\infty(\mathfrak{a})^W$. 
Note that $\mathcal{J}_{\bsk,\emptyset}\,\phi(x)$ for $\phi\in\PW^W$
coincides with \eqref{eqn:jfempty}.

Now we state the main result of this paper:
\begin{thm}[Inversion formula, final form]\label{thm:main}
Let $\bsk\in\mathcal{K}'_1$.
Then for $\phi\in\PW^W$
\[
\mathcal{J}_{\bsk}\,\phi(x)
=\sum_{i=0}^r \mathcal{J}_{\bsk,\varTheta_i}\phi(x)
\quad (x\in \mathfrak{a}_+).
\]
Thus, it holds for 
$f\in C_0^\infty(\mathfrak{a})^W$ that 
\begin{equation*}
f(x)=\sum_{i=0}^r \mathcal{J}_{\bsk,\varTheta_i}\mathcal{F}_\bsk f(x)\quad (x\in \mathfrak{a}).
\end{equation*}
\end{thm}

\begin{proof}
By Theorem \ref{thm:1st} the latter formula follows from the former one. 

First we assume $r=1$ or $\bsk_m>0$.
For $1\leq i\leq r$ and $\xi=(\xi_1,\dots,\xi_i)\in \Dki$ put
\begin{align*}
I_{\varTheta_i}(\xi,\lambda_{\mathfrak{a}_{\varTheta_i}};x)
= & \,\frac{(-1)^i}{(2\pi\sqrt{-1})^{r-i}(r-i)!} \\
&\times \left.\left(\phi(\lambda)\,
{F}_\varXi(\lambda,\bsk;x)
\,{\cc}_\varXi(\lambda,\bsk)^{-1}\right)\right|_{\lambda_{\mathfrak a(\varTheta_i)}=\xi}  \\
& \times \Res_{\lambda_i=\xi_i}\cdots  \Res_{\lambda_2=\xi_2}\Res_{\lambda_1=\xi_1}
\cc(-\lambda,\bsk)^{-1}
\end{align*}
and 
\begin{equation}\label{eqn:int20}
\phi_{\varTheta_i}^\vee(x)= 
 \sum_{\xi\,\in\,\Dki}
\int_{\sqrt{-1}\mathfrak{a}_{\varTheta_i}^*}
I_{\varTheta_i}(\xi,\lambda_{\mathfrak{a}_{\varTheta_i}};x)
\,d\lambda_{\mathfrak{a}_{\varTheta_i}}\quad(x\in\mathfrak a_+).
\end{equation}
The function $\lambda\mapsto\phi(\lambda)\,
{F}_\varXi(\lambda,\bsk;x)
\,{\cc}_\varXi(\lambda,\bsk)^{-1}$ is regular in the region \eqref{eqn:clnegch}.
By \eqref{eqn:thetai} and \eqref{eq:updaownTheta} we have
\begin{multline}\label{eq:RescRes}
\Res_{\lambda_i=\xi_i}\cdots \Res_{\lambda_2=\xi_2}\Res_{\lambda_1=\xi_1}\cc(-\lambda,\bsk)^{-1} \\
 ={\cc}^{\varTheta_i}(-\lambda,\bsk)^{-1}\big|_{\lambda_{\mathfrak a(\varTheta_i)}=\xi}\,
\Res_{\lambda_i=\xi_i}\cdots \Res_{\lambda_2=\xi_2}\Res_{\lambda_1=\xi_1}{\cc}_{\varTheta_i}(-\lambda,\bsk)^{-1}
\end{multline}
 and $\Res_{\lambda_i=\xi_i}\cdots \Res_{\lambda_2=\xi_2}\Res_{\lambda_1
=\xi_1}{\cc}_{\varTheta_i}(-\lambda,\bsk)^{-1}$ is 
a constant that depends only on $\bsk$ and $\xi\in\Dki$.
Hence each integral on the right hand side of \eqref{eqn:int20} converges by 
Lemma~\ref{lem:estcf}, Lemma \ref{lem:FXiest}, 
and Lemma \ref{lem:cTheta}. 
%, and so on.
Let us prove by induction that $\mathcal{J}_\bsk\, \phi(x)-\mathcal{J}_{\bsk,\emptyset}\,\phi(x)
-\sum_{j=1}^{i-1}\phi_{\varTheta_j}^\vee(x)$ 
equals 
\begin{equation}\label{eqn:inti0}
 \sum_{\xi\,\in\,\Dki}
\int_{\check{\xi}_i+\sqrt{-1}\mathfrak{a}_{\varTheta_i}^*}
I_{\varTheta_i}(\xi,\lambda_{\mathfrak{a}_{\varTheta_i}};x)
\,d\lambda_{\mathfrak{a}_{\varTheta_i}}
\end{equation}
for $i=1,\ldots,r$.
Here $\check\xi_{i}=(\xi_{i},\dots,\xi_{i})\in \mathfrak{a}_{\varTheta_i}^*$. 
We already proved the case of $i=1$ by Proposition \ref{prop:second}.
If $r=1$ then we are done.
So let $1\le i<r$ and $\xi=(\xi_1,\dots,\xi_i)\in \Dki$.
Let $\{\xi^{(1)}<\xi^{(2)}<\cdots<\xi^{(k)}=0\}=
\{\xi_{i+1}\in\mathbb R\,;\,(\xi_1,\dots,\xi_i,\xi_{i+1})\in D_\bsk(\varTheta_{i+1})\}\cup\{0\}$.
With a sufficiently small $\varepsilon>0$ 
put
\begin{align*}
\eta^{(1,0)}&=(\xi^{(1)}-\varepsilon,\,\ldots,\,\xi^{(1)}-\varepsilon,\,\xi^{(1)}-\varepsilon,\,\xi^{(1)}-\varepsilon),\\
\eta^{(1,1)}&=(\xi^{(1)}-\varepsilon,\,\ldots,\,\xi^{(1)}-\varepsilon,\,\xi^{(1)}-\varepsilon,\,\xi^{(1)}+\varepsilon),\\
\eta^{(1,2)}&=(\xi^{(1)}-\varepsilon,\,\ldots,\,\xi^{(1)}-\varepsilon,\,\xi^{(1)}+\varepsilon,\,\xi^{(1)}+\varepsilon),\\
&\  \,\vdots\\
\eta^{(1,r-i)}&=(\xi^{(1)}+\varepsilon,\,\ldots,\,\xi^{(1)}+\varepsilon,\,\xi^{(1)}+\varepsilon,\,\xi^{(1)}+\varepsilon),\\
\eta^{(2,0)}&=(\xi^{(2)}-\varepsilon,\,\ldots,\,\xi^{(2)}-\varepsilon,\,\xi^{(2)}-\varepsilon,\,\xi^{(2)}-\varepsilon),\\
\eta^{(2,1)}&=(\xi^{(2)}-\varepsilon,\,\ldots,\,\xi^{(2)}-\varepsilon,\,\xi^{(2)}-\varepsilon,\,\xi^{(2)}+\varepsilon),\\
&\ \,\vdots\\
\eta^{(k-1,r-i)}&=(\xi^{(k-1)}+\varepsilon,\,\ldots,\,\xi^{(k-1)}+\varepsilon,\,\xi^{(k-1)}+\varepsilon,\,\xi^{(k-1)}+\varepsilon),\\
\eta^{(k,0)}&=(\xi^{(k)}-\varepsilon,\,\ldots,\,\xi^{(k)}-\varepsilon,\,\xi^{(k)}-\varepsilon,\,\xi^{(k)}-\varepsilon)\\
&=(-\varepsilon,\,\ldots,\,-\varepsilon,\,-\varepsilon,\,-\varepsilon).
\end{align*}
All these points belong to the region \eqref{eqn:Zreg} with $\lambda_i=\xi_i$.
We shift the domain of the integration
\[
\int_{\check{\xi}_i+\sqrt{-1}\mathfrak{a}_{\varTheta_i}^*}
I_{\varTheta_i}(\xi,\lambda_{\mathfrak{a}_{\varTheta_i}};x)
\,d\lambda_{\mathfrak{a}_{\varTheta_i}}
\]
successively as
\begin{multline*}
\check{\xi}_i+\sqrt{-1}\mathfrak{a}_{\varTheta_i}^*
\,\to\,
\eta^{(1,0)}+\sqrt{-1}\mathfrak{a}_{\varTheta_i}^*
\,\to\,
\eta^{(1,1)}+\sqrt{-1}\mathfrak{a}_{\varTheta_i}^*
\,\to\,\\
\cdots\to\,
\eta^{(k,0)}+\sqrt{-1}\mathfrak{a}_{\varTheta_i}^*
\,\to\, \sqrt{-1}\mathfrak{a}_{\varTheta_i}^*
\end{multline*}
while picking up residues as in the proof of Proposition \ref{prop:second}.
Let $1\le \ell < k$ and $1\le j\le r-i$. 
Using Cauchy's residue theorem and changes of variables
in view of \eqref{eq:WXiinvc}, Lemma \ref{lem:cTheta} and \eqref{eq:RescRes},
we have
\begin{multline*}
\int_{\eta^{(\ell,j)}+\sqrt{-1}\mathfrak{a}_{\varTheta_i}^*}
I_{\varTheta_i}(\xi,\lambda_{\mathfrak{a}_{\varTheta_i}};x)
\,d\lambda_{\mathfrak{a}_{\varTheta_i}}
-
\int_{\eta^{(\ell,j-1)}+\sqrt{-1}\mathfrak{a}_{\varTheta_i}^*}
I_{\varTheta_i}(\xi,\lambda_{\mathfrak{a}_{\varTheta_i}};x)
\,d\lambda_{\mathfrak{a}_{\varTheta_i}}
\\
=2\pi\sqrt{-1}
\int_{(\xi^{(\l)},\ldots,\xi^{(\l)})+\sqrt{-1}\mathfrak{a}_{\varTheta_{i+1}}^*}
\Res_{\lambda_{i+1}=\xi^{(\l)}}
I_{\varTheta_i}(\xi,\lambda_{\mathfrak{a}_{\varTheta_i}};x)
\,d\lambda_{\mathfrak{a}_{\varTheta_{i+1}}}.
\end{multline*}
Note that this is independent of $j$.
Hence
\begin{multline*}
\int_{\check{\xi}_i+\sqrt{-1}\mathfrak{a}_{\varTheta_i}^*}
I_{\varTheta_i}(\xi,\lambda_{\mathfrak{a}_{\varTheta_i}};x)
\,d\lambda_{\mathfrak{a}_{\varTheta_i}}
-
\int_{\sqrt{-1}\mathfrak{a}_{\varTheta_i}^*}
I_{\varTheta_i}(\xi,\lambda_{\mathfrak{a}_{\varTheta_i}};x)
\,d\lambda_{\mathfrak{a}_{\varTheta_i}}
\\
=\sum_{\l=1}^{k-1}
\int_{(\xi^{(\l)},\ldots,\xi^{(\l)})+\sqrt{-1}\mathfrak{a}_{\varTheta_{i+1}}^*}
I_{\varTheta_{i+1}}((\xi,\xi^{(\l)}),\lambda_{\mathfrak{a}_{\varTheta_{i+1}}};x)
\,d\lambda_{\mathfrak{a}_{\varTheta_{i+1}}}.
\end{multline*}
This shows that the difference of \eqref{eqn:inti0} and \eqref{eqn:int20} 
is \eqref{eqn:inti0} with $i$ replaced by $i+1$,
completing the induction.
If $i=r$ then \eqref{eqn:inti0} equals \eqref{eqn:int20}.
Therefore
\[
\mathcal{J}_\bsk\,\phi(x)-\mathcal{J}_{\bsk,\emptyset}\,\phi(x)
=\sum_{i=1}^r \phi_{\varTheta_i}^\vee (x).
\]

Now let $1\le i\le r$ and $\xi\in\Dki$.
From Lemma \ref{lem:crestr} and \eqref{eq:updaownTheta}
it holds that as a meromorphic function on $\mathfrak{a}_{\varTheta_i,\C}^*$
\begin{align*}
{\cc}_\varXi&(\lambda,\bsk)^{-1}\bigr|_{\lambda_{\mathfrak a(\varTheta_i)}=\xi}
\Res_{\lambda_i=\xi_i}\cdots\Res_{\lambda_1=\xi_1}
\cc(-\lambda,\bsk)^{-1}\\
&=\sigma_\xi\bigl(\cc^\varXi(\lambda,\bsk)\bigr)
\Res_{\lambda_i=\xi_i}\cdots\Res_{\lambda_1=\xi_1}
\bigl(
\cc^\varXi(\lambda,\bsk)^{-1}
{\cc}_\varXi(\lambda,\bsk)^{-1}
\cc(-\lambda,\bsk)^{-1}
\bigr)\\
&=\sigma_\xi\bigl(\cc^\varXi(\lambda,\bsk)\bigr)
\bigl(\cc^{\varTheta_i}(\lambda,\bsk)^{-1}\cc^{\varTheta_i}(-\lambda,\bsk)\bigr)^{-1}\bigr|_{\lambda_{\mathfrak a(\varTheta_i)}=\xi}\\
&\hspace{8em}
\times
\Res_{\lambda_i=\xi_i}\cdots\Res_{\lambda_1=\xi_1}
\bigl(\cc_{\varTheta_i}(\lambda,\bsk)^{-1}\cc_{\varTheta_i}(-\lambda,\bsk)^{-1}\bigr).
\end{align*}
Here $\sigma_\xi$ is the restriction operator defined by \eqref{eq:suc_res}.
As we shall see in Proposition~\ref{prop:dtheta},
the factor in the last line reduces to $(-1)^i d_{\varTheta_i} (\xi,\bsk)$.
Furthermore $\cc^{\varTheta_i}(\lambda,\bsk)\,\cc^{\varTheta_i}(-\lambda,\bsk)$
is $W(\varTheta_i)$-invariant and restricts to
$|\cc^{\varTheta_i}(\lambda,\bsk)|^2$ on $\xi+\sqrt{-1}\mathfrak{a}_{\varTheta_i}^*$ 
 (cf. \cite[Lemma~6.6]{Sh94}). 
Hence by changes of variables and Proposition~\ref{prop:xitempered}
\begin{multline*}
\int_{\sqrt{-1}\mathfrak{a}_{\varTheta_i}^*}
I_{\varTheta_i}(\xi,\lambda_{\mathfrak{a}_{\varTheta_i}};x)
\,d\lambda_{\mathfrak{a}_{\varTheta_i}}\\
=
\frac{d_{\varTheta_i} (\xi,\bsk)}{|W(\varTheta_i)|}
\int_{\sqrt{-1}\mathfrak{a}_{\varTheta_i}^*}
\left. \frac{\phi(\lambda)F(\lambda,\bsk;x)}
{|\cc^{\varTheta_i}(\lambda,\bsk)|^2}
\right|_{\lambda_{\mathfrak a(\varTheta_i)}=\xi}
d\mu_{\varTheta_i}(\lambda_{\mathfrak a_{\varTheta_i}}).
\end{multline*}
Thus
$\phi_{\varTheta_i}^\vee(x)=\mathcal{J}_{\bsk,\varTheta_i}\,\phi(x)$
for $x\in\mathfrak a_+$ and $i=1,\ldots,r$,
proving the theorem under the assumption that $r=1$ or $\bsk_m>0$.

Next, let us assume $r>1$ and $\bsk_m=0$.
In this case Proposition \ref{prop:second} is not applicable.
Instead, we use the result in the case of $r=1$ (the case of the Jacobi transform)
that is shown above.
Let $F_0(\lambda_1;x_1)$, $\varPhi_0(\lambda_1;x_1)$, $\cc_0(\lambda_1)$
and $d_0(\lambda_1)$
be respectively
$F(\lambda,\bsk;x)$, $\varPhi(\lambda,\bsk;x)$, $\cc(\lambda,\bsk)$
and $d_{\mathcal B}(\lambda,\bsk)$
in the case of $r=1$.
Then we have
\begin{align*}
F(\lambda,\bsk;x)&=\frac1{r!}\sum_{s\in\mathfrak S_r}\prod_{j=1}^r F_0(\lambda_j;x_{s(j)}),
&\varPhi(\lambda,\bsk;x)&=\prod_{j=1}^r \varPhi_0(\lambda_j;x_j),\\
d_{\varTheta_i}(\lambda_{\mathfrak{a}(\varTheta_i)},\bsk)&=\frac{(i!)^2}{\bigl|W_{\varTheta_i}^{\lambda_{\mathfrak{a}(\varTheta_i)}}\bigr|}\prod_{j=1}^i d_0(\lambda_j),
&\cc^{\varTheta_i}(\lambda,\bsk)&=\frac{i!}{r!}\prod_{j=i+1}^r \cc_0(\lambda_j).
\end{align*}
The first and second formulas are respectively (2-7) and (2-4) of \cite{Sh08}.
To deduce the third and fourth formulas use \eqref{eqn:crho}.
Hence if $x=(x_1,\ldots,x_r)\in\mathfrak a_+$, $\eta\ll0$ and $s\in \mathfrak S_r$ then
$\mathcal{J}_\bsk\, \phi(x)$ is equal to
\begin{align}\label{eq:ititgl}
&\int_{(\eta,\ldots,\eta)+\sqrt{-1}\mathfrak a^*} \phi(\lambda)\,
\varPhi((\lambda_{s^{-1}(1)},\ldots,\lambda_{s^{-1}(r)}),\bsk,x)\cc(-\lambda,\bsk)^{-1}d\mu(\lambda)\\
&\,=
\int_{(\eta,\ldots,\eta)+\sqrt{-1}\mathfrak a^*} \phi(\lambda)\,
\varPhi(\lambda,\bsk,(x_{s(1)},\ldots,x_{s(r)}))\cc(-\lambda,\bsk)^{-1}d\mu(\lambda) \notag\\
&\,=\frac{r!}{(2\pi\sqrt{-1})^r}\int_{\eta-\sqrt{-1}\infty}^{\eta+\sqrt{-1}\infty}
\frac{\varPhi_0(\lambda_1;x_{s(1)})d\lambda_1}{\cc_0(-\lambda_1)}
\int_{\eta-\sqrt{-1}\infty}^{\eta+\sqrt{-1}\infty}
\frac{\varPhi_0(\lambda_2;x_{s(2)})d\lambda_2}{\cc_0(-\lambda_2)} \notag\\
&\hspace{10em} \cdots
\int_{\eta-\sqrt{-1}\infty}^{\eta+\sqrt{-1}\infty}
\frac{\phi(\lambda_1,\ldots,\lambda_r)\varPhi_0(\lambda_r;x_{s(r)})d\lambda_r}{\cc_0(-\lambda_r)}. \notag
\end{align}
Each step in the iterated integral is the inverse Jacobi transform
of a Paley-Wiener function.
If $\psi(\lambda_j)$ is a Paley-Wiener function of one variable,
then by the result in the case of $r=1$ we have
\begin{align*}
&\frac1{2\pi\sqrt{-1}}\int_{\eta-\sqrt{-1}\infty}^{\eta+\sqrt{-1}\infty}
\frac{\psi(\lambda_j)\,\varPhi_0(\lambda_j;x_{s(j)})d\lambda_j}{\cc_0(-\lambda_j)}\\
&=
\frac1{4\pi\sqrt{-1}}\int_{\sqrt{-1}\mathbb R}
\frac{\psi(\lambda_j) F_0(\lambda_j;x_{s(j)})d\lambda_j}{|\cc_0(\lambda_j)|^2}
+\!\!\!\sum_{\lambda_j\in D_\bsk(\varTheta_1)} \!\!\!
d_0(\lambda_j)\psi(\lambda_j) F_0(\lambda_j;x_{s(j)}).
\end{align*}
Therefore \eqref{eq:ititgl} becomes
\begin{multline}\label{eq:invk0}
\sum_{i=0}^r
\sum_{\substack{J=\{j_1<\cdots<j_i\} \\ \subset \{1,\ldots,r\}}}
\sum_{\substack{(\lambda_{j_1},\lambda_{j_2},\ldots,\lambda_{j_i}) \\ \in D_\bsk(\varTheta_1)^i}}
\frac{r!\prod_{j\in J}d_0(\lambda_j)}{(4\pi\sqrt{-1})^{r-i}}\\
\times\int_{\sqrt{-1}\mathbb R^{r-i}}
\phi(\lambda)\prod_{j=1}^r F_0(\lambda_j;x_{s(j)})
\prod_{j\notin J}\frac{d\lambda_j}{|\cc_0(\lambda_j)|^2}.
\end{multline}
Since $\mathcal{J}_\bsk\, \phi(x)$ (which is equal to \eqref{eq:invk0}) is independent of $s \in \mathfrak S_r$, we can replace $\prod_{j=1}^r F_0(\lambda_j;x_{s(j)})$ in \eqref{eq:invk0}
with $F(\lambda,\bsk;x)$.
By changes of variables, \eqref{eq:invk0} reduces to
\begin{align*}
&\sum_{i=0}^r
\sum_{\lambda_{\mathfrak{a}(\varTheta_i)}\in D_\bsk(\varTheta_i)}
\frac{(i!)^2\prod_{j=1}^i d_0(\lambda_j)}{(2\pi\sqrt{-1})^{r-i}|W(\varTheta_i)|\bigl|W_{\varTheta_i}^{\lambda_{\mathfrak{a}(\varTheta_i)}}\bigr|}\\
&\hspace{10em}\times\int_{\sqrt{-1}\mathfrak{a}_{\varTheta_i}^*}
\phi(\lambda)F(\lambda,\bsk;x)\biggl(\frac{r!}{i!}\biggr)^2
\prod_{j=i+1}^r\frac{d\lambda_j}{|\cc_0(\lambda_j)|^2}\\
&=\sum_{i=0}^r \mathcal{J}_{\bsk,\varTheta_i}\phi(x).\qedhere
\end{align*}
\end{proof}

\begin{prop}\label{prop:dtheta}
Assume $\bsk\in\mathcal{K}'_1$ and 
let $1\leq i\leq r$. 
If $r> 1$ then we further assume $\bsk_m>0$.
For $\xi\in\Dki$, 
\begin{equation}\label{eqn:fdtheta}
(-1)^i
\Res_{\lambda_i=\xi_i}\cdots\Res_{\lambda_1=\xi_1}
\bigl(\cc_{\varTheta_i}(\lambda,\bsk)^{-1}\cc_{\varTheta_i}(-\lambda,\bsk)^{-1}\bigr)
=d_{\varTheta_i} (\xi,\bsk).
\end{equation}
\end{prop}

\begin{proof}
By \eqref{eqn:tcfbc1} and the following formulas
\begin{align}
\label{eqn:gamma0}
& \varGamma(z+1)=z\varGamma(z), \\
& \varGamma(z)\varGamma(1-z)=\frac{\pi}{\sin \pi z}
\label{eqn:gamma2}
\end{align}
for the Gamma function, we have
\begin{align}\label{eq:cici}
\tilde{\cc}_i(\lambda,&\bsk)^{-1}\tilde{\cc}_i(-\lambda,\bsk)^{-1}\\
&=\frac{2^{\al-\be-2}\lambda_i\sin(-\pi\lambda_i)}{\sin\frac{\pi}{2}(\lambda_i+\al-|\be|+1)\sin\frac{\pi}{2}(-\lambda_i+\al-|\be|+1)}\,p(\lambda_i) \notag\\
&=\frac{2^{\al-\be-2}\lambda_i\sin\pi\lambda_i}{\sin\frac{\pi}{2}(\lambda_i+\xi_1)\sin\frac{\pi}{2}(\lambda_i-\xi_1)}\,p(\lambda_i) \notag\\
\intertext{with}
p(z)&:=\frac{\varGamma\big(\frac12(z+\al+|\be|+1)\big)
\varGamma\big(\frac12(-z+\al+|\be|+1)\big)}{\varGamma\big(\frac12(z-\al+|\be|+1)\big)
\varGamma\big(\frac12(-z-\al+|\be|+1)\big)}.\notag
\end{align}
Here the second equality in \eqref{eq:cici} holds
since $\xi_1-\al+|\be|-1\in2\mathbb{N}$.
Hence if $i=1$ then
\begin{align*}
-\Res_{\lambda_1=\xi_1}
\bigl(&\tilde{\cc}_1 (\lambda,\bsk)^{-1} \tilde{\cc}_1(-\lambda,\bsk)^{-1}\bigl)\\
&=
-2^{\al-\be-2}\xi_1\, p(\xi_1) \cdot
\left.\frac{\sin\pi\lambda_1}{\sin\frac{\pi}{2}(\lambda_1+\xi_1)}\right|_{\lambda_1=\xi_1}\cdot
\Res_{\lambda_1=\xi_1}\frac1{\sin\frac{\pi}{2}(\lambda_1-\xi_1)}\\
&=\frac{-2^{\al-\be-1}\xi_1}{\pi}p(\xi_1).
\end{align*}
This proves \eqref{eqn:fdtheta} for $i=1$. 

Now let $i>1$.
We will prove \eqref{eqn:fdtheta}
assuming the equality holds for $i-1$.
Let $\xi=(\xi',\xi_{i})\in \Dki$ 
with $\xi'\in D_{\bsk}(\varTheta_{i-1})$.
It is sufficient to show 
\begin{align}\label{eqn:resind}
& -\frac{
\tilde{\cc}_{\varTheta_{i}}(\rho(\bsk),\bsk)^{-2}\,d_{\varTheta_{i}} (\xi,\bsk)}
{\tilde{\cc}_{\varTheta_{i-1}}(\rho(\bsk),\bsk)^{-2}\,d_{\varTheta_{i-1}} (\xi',\bsk)} \\
& \,\, =
\Res_{\lambda_{i}=\xi_{i}}\!\prod_{\alpha\,\in \,\langle\varTheta_{i}\rangle^+\,\setminus
\, \langle\varTheta_{i-1}\rangle}\!
\tilde{\cc}_\alpha ((\xi',\lambda_{i}), \bsk)^{-1}\tilde{\cc}_\alpha((-\xi',-\lambda_{i}),\bsk)^{-1}.
\notag
\end{align}
The set 
$\langle\varTheta_{i}\rangle^+\,\setminus\, \langle\varTheta_{i-1}\rangle$ consists of 
roots $2\beta_{i}$, $\beta_{i}$, $\beta_{i}\pm \beta_j$ $(1\leq j< i)$. 
By \eqref{eqn:tca}, \eqref{eqn:gamma0} and \eqref{eqn:gamma2}
we have for $j=1,\ldots,i-1$
\begin{align}\label{eq:cbb4}
\tilde{\cc}_{\beta_i-\beta_j}& ((\xi',\lambda_{i}), \bsk)^{-1}
\tilde{\cc}_{\beta_i+\beta_j}((\xi',\lambda_{i}),\bsk)^{-1}\\
\times&
\tilde{\cc}_{\beta_i-\beta_j} ((-\xi',-\lambda_{i}),\bsk), \bsk)^{-1}
\tilde{\cc}_{\beta_i+\beta_j}((-\xi',-\lambda_{i}),\bsk)^{-1} \notag\\
&=\frac{-
\sin\frac{\pi}{2}(\lambda_i+\xi_j)\sin\frac{\pi}{2}(\lambda_i-\xi_j)}
{\sin\frac{\pi}{2}(\lambda_i+\xi_j+2\bsk_m)\sin\frac{\pi}{2}(-\lambda_i+\xi_j+2\bsk_m)}\,q(\lambda_i,\xi_j) \notag\\
&=\frac{
\sin\frac{\pi}{2}(\lambda_i+\xi_j)\sin\frac{\pi}{2}(\lambda_i-\xi_j)}
{\sin\frac{\pi}{2}(\lambda_i+\xi_{j+1})\sin\frac{\pi}{2}(\lambda_i-\xi_{j+1})}\,q(\lambda_i,\xi_j) \notag\\
\intertext{with}
q(z,w)&:=\frac{(w^2-z^2) 
\varGamma\big(\frac12(z-w+2\bsk_m)\big)\varGamma\big(\frac12(-z-w+2\bsk_m)\big)}
{4\varGamma\big(\frac12(z-w-2\bsk_m+2)\big)\varGamma\big(\frac12(-z-w-2\bsk_m+2)\big)}.\notag
\end{align}
Here the second equality in \eqref{eq:cbb4} holds
since $\xi_{j+1}-\xi_{j}-2\bsk_m\in2\mathbb{N}$.
Combining \eqref{eq:cbb4} with \eqref{eq:cici} we have
\begin{multline*}
\prod_{\alpha\,\in \,\langle\varTheta_{i}\rangle^+\,\setminus
\, \langle\varTheta_{i-1}\rangle}\!
\tilde{\cc}_\alpha ((\xi',\lambda_{i}), \bsk)^{-1}\tilde{\cc}_\alpha((-\xi',-\lambda_{i}),\bsk)^{-1}\\
=\frac{2^{\al-\be-2}\lambda_i\sin\pi\lambda_i}{\sin\frac{\pi}{2}(\lambda_i+\xi_i)\sin\frac{\pi}{2}(\lambda_i-\xi_i)}\,p(\lambda_i)\prod_{j=1}^{i-1} q(\lambda_i,\xi_j),
\end{multline*}
whose residue at $\lambda_i=\xi_i$ equals
\begin{multline*}
2^{\al-\be-2}\xi_i\, p(\xi_i)\prod_{j=1}^{i-1} q(\xi_i,\xi_j) \cdot
\left.\frac{\sin\pi\lambda_i}{\sin\frac{\pi}{2}(\lambda_i+\xi_i)}\right|_{\lambda_i=\xi_i}\cdot
\Res_{\lambda_i=\xi_i}\frac1{\sin\frac{\pi}{2}(\lambda_i-\xi_i)}\\
=\frac{2^{\al-\be-1}\xi_i}{\pi}\,p(\xi_i)\prod_{j=1}^{i-1} q(\xi_i,\xi_j).
\end{multline*}
Thus \eqref{eqn:resind} is proved.
\end{proof}

\begin{thm}[Paley-Wiener theorem]\label{thm:pw}
Let $\bsk\in\mathcal{K}'_1$. 
Then the hypergeometric Fourier transform $\mathcal{F}_\bsk$ is a bijection of 
$C_0^\infty(\mathfrak{a})^W$ onto $\PW^W$. 
\end{thm}
\begin{proof}
Let $\phi\in\PW^W$.
By Proposition \ref{prop:J_kbdd} and Theorem \ref{thm:main},
$\mathcal J_\bsk\phi(x)\in C^\infty(\mathfrak a_+)$ extends to 
a function in $C^\infty_0(\mathfrak a)^W$.
Thus it only remains to prove that $\mathcal{F}_\bsk\mathcal J_\bsk\phi(\lambda)=\phi(\lambda)$.
By \cite[Theorem~9.13]{Op:Cherednik} the equality holds if $\bsk\in\mathcal K_+$.
The general result follows from this because
by \eqref{eq:HGFshift1}, \eqref{eq:HGFshift2}, \eqref{eq:WPOshift1} and \eqref{eq:WPOshift2}
it holds that $\mathcal{F}_\bsk\mathcal J_\bsk\phi=\mathcal{F}_{\tilde\bsk}\mathcal J_{\tilde\bsk}\phi=\mathcal{F}_{\bsk+\boldsymbol 1}\mathcal J_{\bsk+\boldsymbol 1}\phi$
for any $\bsk\in\mathcal{K}'_1$.
\end{proof}

Now we define for $\phi\in \mathcal S_\bsk$
\[
\mathcal{J}_{\bsk}\,\phi(x)
=\sum_{i=0}^r \mathcal{J}_{\bsk,\varTheta_i}\phi(x)
\in C^\infty(\mathfrak a)^W.
\]
Let $\nu_\bsk$ denote the measure on $\bigsqcup_{i=0}^r (\Dki+\sqrt{-1}\Cl(\mathfrak{a}_{\varTheta_i,+}^*))$
that coincides with $\nu_{\bsk,\varTheta_i}$ on each $\Dki+\sqrt{-1}\Cl(\mathfrak{a}_{\varTheta_i,+}^*)$.
By Lemma \ref{lem:estcf} and Lemma \ref{lem:PWinS},
both $\PW^W$ and $\mathcal S_\bsk$ are naturally identified with
dense subspaces of
$L^2(\bigsqcup_{i=0}^r (\Dki+\sqrt{-1}\Cl(\mathfrak{a}_{\varTheta_i,+}^*))\,;d\nu_\bsk)$
by restriction.

\begin{lem}\label{lem:FJadjunction}
Suppose $\phi\in\mathcal S_\bsk$ and $f\in C_0^\infty(\mathfrak{a})^W$.
Then
\[
\frac{1}{|W|}\int_{\mathfrak{a}} \mathcal J_\bsk\phi(x) \overline{f(x)} \delta_\bsk(x)dx
=\int_{\bigsqcup_{i=0}^r (\Dki+\sqrt{-1}\Cl(\mathfrak{a}_{\varTheta_i,+}^*))}
\phi(\lambda)\, \overline{\mathcal{F}_\bsk f(\lambda)}\,d\nu_\bsk(\lambda).
\]
\end{lem}
\begin{proof}
The result follows from Lemma~\ref{lem:estcf}, Lemma~\ref{lem:est} and Fubini's theorem.
\end{proof}

\begin{thm}[Plancherel theorem]\label{thm:plancherel}
Let $\bsk\in\mathcal{K}'_1$. Then for 
$f\in C_0^\infty(\mathfrak{a})^W$, 
\begin{equation}\label{eqn:plancherel}
\frac{1}{|W|}\int_{\mathfrak{a}}|f(x)|^2\delta_\bsk(x)dx
=\int_{\bigsqcup_{i=0}^r (\Dki+\sqrt{-1}\Cl(\mathfrak{a}_{\varTheta_i,+}^*))}
|\mathcal{F}_\bsk f(\lambda)|^2\,d\nu_\bsk(\lambda).
\end{equation}
Moreover, the hypergeometric Fourier transform $\mathcal{F}_\bsk$ extends to an isometry of 
$L^2(\mathfrak{a}\,; \frac{1}{|W|}\delta_\bsk(x)dx)^W$ onto 
$L^2(\bigsqcup_{i=0}^r (\Dki+\sqrt{-1}\Cl(\mathfrak{a}_{\varTheta_i,+}^*))\,;d\nu_\bsk)$. 
For any $\phi\in\mathcal S_\bsk$ it holds that
$\mathcal{F}_\bsk^{-1}\phi=\mathcal J_\bsk\phi$.
\end{thm}
\begin{proof}
Applying Lemma \ref{lem:FJadjunction} to $\phi=\mathcal F_\bsk f$
we obtain \eqref{eqn:plancherel} by Theorem~\ref{thm:main}.
Thus $\mathcal{F}_\bsk$ extends to an isometry of 
$L^2(\mathfrak{a}\,; \frac{1}{|W|}\delta_\bsk(x)dx)^W$ into  
$L^2(\bigsqcup_{i=0}^r (\Dki+\sqrt{-1}\Cl(\mathfrak{a}_{\varTheta_i}^*,+))\,;d\nu_\bsk)$.
This is onto since the image is dense. 

Now let $\phi\in\mathcal S_\bsk$ and put $g=\mathcal{F}_\bsk^{-1}\phi \in L^2(\mathfrak{a}\,; \frac{1}{|W|}\delta_\bsk(x)dx)^W$.
Then for any $f\in C_0^\infty(\mathfrak{a})^W$
\[
(g,f)=(\phi,\mathcal F_\bsk f)=\int_{\bigsqcup_{i=0}^r (\Dki+\sqrt{-1}\Cl(\mathfrak{a}_{\varTheta_i,+}^*))}
\phi(\lambda)\, \overline{\mathcal{F}_\bsk f(\lambda)}\,d\nu_\bsk(\lambda).
\]
Here $(\cdot,\cdot)$ stands for the inner products of two Hilbert spaces.
Comparing this with Lemma \ref{lem:FJadjunction}, one sees $(g-\mathcal J_\bsk\phi)\delta_\bsk=0$
as a distribution on $\mathfrak a$.
Thus $g(x)=\mathcal J_\bsk\phi(x)$ almost everywhere.
\end{proof}

\begin{rem}
Let $\tilde\nu_\bsk$ denote the $W$-invariant measure on $\bigcup \varPi_\bsk$
whose restriction to $\lambda_{\mathfrak{a}(\varTheta_i)}+\sqrt{-1}\mathfrak{a}_{\varTheta_i}^*$
($0\le i\le r, \lambda_{\mathfrak{a}(\varTheta_i)}\in \Dki$) is
\[
\frac{|W^{\lambda_{\mathfrak{a}(\varTheta_i)}}|}{|W||W(\varTheta_i)|}\nu_\bsk
=
\frac{|W^{\lambda_{\mathfrak{a}(\varTheta_i)}}_{\varTheta_i}|}{|W|}\nu_\bsk.
\]
By Lemma \ref{lem:specrepr}, $L^2(\bigsqcup_{i=0}^r (\Dki+\sqrt{-1}\Cl(\mathfrak{a}_{\varTheta_i,+}^*))\,;d\nu_\bsk)$ is naturally identified with
$L^2(\bigcup \varPi_\bsk\,;d\tilde \nu_\bsk)^W$.

We expect $\mathcal J_\bsk\mathcal S_\bsk$
would coincide with the Schwartz space $\mathcal C(\mathfrak a,\bsk)^W$
introduced in \cite{DT99}.
In fact, Delorme \cite{DT99} proved, for each reduced root system,
that
$\mathcal F_\bsk \mathcal C(\mathfrak a,\bsk)^W$
coincides with a version of the Schwartz space which is defined in a similar way to $\mathcal S_\bsk$.
\end{rem}

\begin{cor}\label{cor:lthg}
Let $\bsk\in\mathcal{K}'_1$.  
Then the hypergeometric function $F(\lambda,\bsk)$ is square integrable if and only if 
$\lambda\in W\Dkr$. 
Each square integrable hypergeometric 
function is of the form 
\begin{equation}\label{eqn:sqinthg}
F(\lambda,\bsk)=\cc(\lambda,\bsk)\,\varPhi(\lambda,\bsk)\quad \text{for some }\,\,
\lambda\in \Dkr.
\end{equation}
For any $\lambda\in \Dkr$, 
\begin{equation}\label{eqn:dsip}
\frac{1}{|W|}\int_\mathfrak{a} F(\lambda,\bsk;x)^2\,\delta_\bsk(x)\,dx=\frac{1}{d_\B(\lambda,\bsk)}.
\end{equation}
For any $\lambda,\,\mu\in \Dkr$ with $\lambda\not=\mu$, 
\begin{equation}\label{eqn:dsoth}
\frac{1}{|W|}\int_\mathfrak{a} F(\lambda,\bsk;x)F(\mu,\bsk;x)\,\delta_\bsk(x)\,dx=0.
\end{equation}
\end{cor}
\begin{proof}
Let $\xi\in \Dkr$ and 
let $\phi_\xi(\lambda)\in \mathcal S_\bsk$ be the function which takes $1$ at $\lambda=\xi$
and $0$ elsewhere.
Then $\mathcal J_\bsk \phi_\xi(x)=d_{\B}(\xi,\bsk)\,F(\xi,\bsk; x)$.
Hence by Theorem \ref{thm:plancherel},
$F(\xi,\bsk,x)$ is square integrable and $\mathcal F_\bsk F(\xi,\bsk)=d_{\B}(\xi,\bsk)^{-1} \phi_\xi$.
(We already know the square integrability by Corollary~\ref{cor:tempered2}.)
Now \eqref{eqn:dsip} and \eqref{eqn:dsoth} are immediate.

Let us prove $F(\xi,\bsk)$'s exhaust the square integrable hypergeometric functions.
Assume that $F(\mu,\bsk)$ is  square integrable for some $\mu\in\acs$.
Let $p\in S(\mathfrak a_{\mathbb C})^W$ and let $f\in C_0^\infty(\mathfrak{a})^W$.
By \eqref{eq:CFlF}, \eqref{eq:adjunction}, \eqref{eq:dualop} and Theorem \ref{thm:plancherel} we have
\begin{align*}
&\int_{\bigsqcup_{i=0}^r (\Dki+\sqrt{-1}\Cl(\mathfrak{a}_{\varTheta_i,+}^*))}
p(\mu) \mathcal{F}_\bsk F(\mu,\bsk)(\lambda)
\, \overline{\mathcal{F}_\bsk f(\lambda)}\,d\nu_\bsk(\lambda)\\
&=\frac{p(\mu)}{|W|} \int_{\mathfrak a} F(\mu,\bsk;x) \overline{f(x)}\delta_\bsk(x)dx \\
&=
\frac{1}{|W|} \int_{\mathfrak a} (T(\bsk,p)F(\mu,\bsk;x)) \overline{f(x)}\delta_\bsk(x)dx \\
&=
\frac{1}{|W|} \int_{\mathfrak a} F(\mu,\bsk;x) (T(\bsk,p)\overline{f(x)})\delta_\bsk(x)dx \\
&=
\int_{\bigsqcup_{i=0}^r (\Dki+\sqrt{-1}\Cl(\mathfrak{a}_{\varTheta_i,+}^*))}
\mathcal{F}_\bsk F(\mu,\bsk)(\lambda)
\, \overline{\mathcal{F}_\bsk T(\bsk,p^\vee )f(\lambda)}\,d\nu_\bsk(\lambda)\\
&=
\int_{\bigsqcup_{i=0}^r (\Dki+\sqrt{-1}\Cl(\mathfrak{a}_{\varTheta_i,+}^*))}
p(\lambda) \mathcal{F}_\bsk F(\mu,\bsk)(\lambda)
\, \overline{\mathcal{F}_\bsk f(\lambda)}\,d\nu_\bsk(\lambda).
\end{align*}
Here $p^\vee(\lambda)=\overline{p(\bar \lambda)}$.
Note that $p(\bar \lambda)=p(\lambda)$ on $\bigsqcup_{i=0}^r (\Dki+\sqrt{-1}\Cl(\mathfrak{a}_{\varTheta_i,+}^*))$
by the $W$-invariance.
Thus by Lemma \ref{lem:PWinS}, $(p(\mu)-p(\lambda))\mathcal{F}_\bsk F(\mu,\bsk)(\lambda)=0$
except on a null set of $\nu_\bsk$.
Since $p$ can be chosen arbitrarily, this is possible only when $\mu=w\xi$ for some $w\in W$
and $\xi\in \Dkr$.

The second statement of the corollary follows from Theorem \ref{thm:temperedtheta1}.
\end{proof}

\begin{rem}
The square integrable hypergeometric functions are analytic continuation of the Jacobi polynomials. 
This fact was observed by  \cite[Remark~5.12]{Sh94} for the group case and 
mentioned without proof in \cite[\S 6]{BHO}. 

If $\be<0$, then $\lambda\in \Dkr$ if and only if $\lambda_r<0$ and 
$\mu=\lambda-\rho(\bsk)$ satisfies
\[
\mu_1\in 2\mathbb{N},\quad \mu_{i+1}-\mu_i\in 2\mathbb{N}\,\,(1\leq i\leq r-1). 
\]
Thus,  $\Dkr$ can be written as
\begin{align}
\Dkr=\{\mu+\rho(\bsk)\in\mathfrak{a}^*\,;\, & \mu+\rho(\bsk)\in  -w_\varXi^*\, \mathfrak{a}_+^*\,\,\\
& 
\text{ and }\,
\langle\mu,\alpha^\vee\rangle
\in \mathbb{N} \,\,\,\text{for all}\,\,\,\alpha\in\Rp\}.\notag
\end{align}
 Here $w_\varXi^*$ denotes the longest element of $W_\varXi=\mathfrak{S}_r$.

If $\bsk\in\mathcal{K}_+$
and 
$\mu\in\mathfrak{a}^*$ satisfies $\langle\mu,\alpha^\vee\rangle\in \mathbb{N}$ for all $
\alpha\in\Rp$, then $\varPhi(\mu+\rho(\bsk),\bsk)=P(\mu,\bsk)$, 
which  is the 
Jacobi polynomial of the highest weight $\mu$ (cf. \cite{HO87, Hec87, {Hec94}}). 
Note $\bsk\in\mathcal{K}_+$ if and only if $\al-\be\geq 0,\,\be\geq -\frac12$, and $\bsk_m\geq 0$. 
The function $F(\mu+\rho(\bsk),\bsk)=\cc(\mu+\rho(\bsk),\bsk)P(\mu,\bsk)$ forms a 
one parameter family of polynomials with analytic parameter $\be$.
If 
\[
\be<-\al-2(r-1)\bsk_m-1,
\]
then $\Dkr\not=\emptyset$ and $F(\mu+\rho(\bsk),\bsk)\,\,(\mu+\rho(\bsk)\in \Dkr)$ 
is square integrable. On the one hand, for $\bsk\in \mathcal{K}_+$ the $L^2$-norm of 
the constant multiple $\cc(\mu+\rho(\bsk),\bsk)P(\mu,\bsk)$ of the Jacobi polynomial is an integral 
of $\cc(\mu+\rho(\bsk),\bsk)^2 P(\mu,\bsk)^2\,\delta_\bsk$ 
over a compact torus and its explicit formula is given by  \cite[Corollary~3.5.3]{Hec94}. 
On the other hand, for $\bsk$ satisfying \eqref{eqn:mcond2} and $\mu+\rho(\bsk)\in \Dkr$ 
the $L^2$-norm of $F(\mu+\rho(\bsk),\bsk)$ 
 is an integral 
of $\frac{1}{|W|}F(\mu+\rho(\bsk),\bsk)^2\,\delta_\bsk$  over $\mathfrak{a}$ and its explicit formula is given by 
\eqref{eqn:dsip} and \eqref{eqn:fdefd}. 
Comparing these formulas we can observe that the two norms coincide up to a constant multiple that depends only on $\bsk$. 

We can reduce the case of $\be>0$ to the case of $\be<0$ by \eqref{eq:ksymm}.
\end{rem}

\subsection*{Acknowledgments} 
The authors would like to thank Professor Toshio Oshima for helpful discussions on the subject matter 
of this paper. 
The authors are grateful to the anonymous referee for 
%his careful reading of our manuscript and 
valuable comments and suggestions.

\end{document}